\documentclass[11pt,a4paper]{scrartcl}

\usepackage[mathfont=libertinus]{xelatex-arxiv}
\microtypesetup{tracking=false,kerning=false}  

\title{The number of particles in activated random walk on the complete graph%
  \thanks{L.T.\ thanks the German Research Foundation (project number 444084038,
    priority program SPP2265) for financial support. The research of C.M.\ was
    in large parts funded by German Research Foundation (project number 443916008,
    priority program SPP2265).}}

\author{%
  Antal A.~J\'arai\thanks{University of Bath, UK.
    \texttt{a.jarai@bath.ac.uk}}\,\orcidlink{0000-0003-3522-498X}
  \and
  Christian M\"onch\thanks{Unaffiliated.
    \texttt{cmoench25@gmail.com}}\,\orcidlink{0000-0002-6531-6482}
  \and
  Lorenzo Taggi\thanks{Sapienza Universit\`a di Roma, Italy.
    \texttt{lorenzo.taggi@uniroma1.it}}\,\orcidlink{0000-0002-7085-9764}
}

\date{arXiv:2304.10169}


\newcommand{\cF}{\mathcal{F}}

\renewcommand{\P}{P_N}

\newcommand{\E}{E_N}
\newcommand{\Var}{\mathrm{Var}}

\newcommand{\eps}{\varepsilon}

\def\bfone{\mathbf{1}}

\newcommand{\eqn}[2]{\begin{equation}\label{#1}#2\end{equation}}
\newcommand{\eqnst}[1]{\begin{equation*}#1\end{equation*}}
\newcommand{\eqnspl}[2]{\begin{equation}\begin{split}\label{#1}%
			#2\end{split}\end{equation}}
\newcommand{\eqnsplst}[1]{\begin{equation*}\begin{split}%
			#1\end{split}\end{equation*}}

\begin{document}
\maketitle
\raggedbottom

\begin{abstract}
	We consider an elementary model for self-organised criticality,
	the activated random walk on the complete graph.
	We introduce a discrete time Markov chain as follows.  At each time step
	we add an active particle at a random vertex and let the system stabilise following the activated random walk dynamics,
	obtaining a particle configuration with all sleeping particles.
	Particles visiting a boundary vertex are removed from the system.
	We characterise the support of the stationary distribution of this Markov chain,
	showing that, with high probability, the number of particles
	concentrates around the value  $\rho_c N + a \sqrt{N \log N}$
	{with fluctuations} of order at most $o( \sqrt{N \log N })$,
	where $N$ is the number of vertices. Due to the mean-field nature of the model, we are able to determine precisely the critical density $\rho_c= \frac{\lambda}{1+\lambda}$, where $\lambda$ is the sleeping rate, as well as the constant $a =  \sqrt{\lambda} / (1 + \lambda) $ characterising the lower order shift. {Our approach utilises results about super-martingales associated with activated random walk that are of independent interest.}
\end{abstract}

\section{Introduction and main results}
Self-organized criticality (SOC) is a property of dynamical systems that have a critical point as an attractor. The concept was put forward by
Bak, Tang and Wiesenfeld \cite{Bak}, and has been applied across different fields to explain a variety of physical phenomena. 

Together with models for `sandpile'-type dynamics, like the Abelian sandpile, the stochastic sandpile model, and the Bak-Sneppen evolution model, the activated random walk (ARW) is one of the models which has been proposed to provide a mathematically rigorous description of SOC \cite{Dickman, Rolla}. It is a conservative particle system with two types of particles, active (A) and sleeping (S). Active particles perform a continuous time simple random walk on a graph. In addition, to each active particle is associated an independent Poisson clock with rate $\lambda$. Whenever a particle's clock rings, the particle turns into the S-state if it does not share its location with another active particle, in which case nothing happens. Each S-particle does not move and switches into the A-state whenever its location is visited by another A-particle. Consequently, any configuration with only S-type particles is an absorbing state for the induced Markovian dynamics.

As highlighted in the recent survey of Levine and Silvestri \cite{LevineSilvestri}, 
in contrast with the `sandpile'-type dynamics of the original example proposed in \cite{Bak}, to which 
a lot of attention has been given \cite{Dharsurvey,Jsurvey}, the ARW model 
appears to have better universality properties, and therefore can serve as a more suitable
model of SOC. Indeed, key features of ARW are believed to be independent of the underlying lattice 
or details of initial conditions, one rigorous example being the remarkable result of 
Rolla, Sidoravicius and Zindy \cite{Rolla3} showing that stabilisability properties 
only depend on the particle density.

Previous research on the activated random walk model mostly considers the dynamics on $\mathbb{Z}^d$ or, more generally, on infinite vertex-transitive graphs. In this case the model does not exhibit self-organised critical behaviour, instead it undergoes an ordinary \emph{absorbing-state phase transition.} More precisely, at time zero, take the particle configuration to be distributed according to a product of Poisson distributions with parameter $\rho \in (0, \infty)$, the \emph{particle density.} If the particle density is below  a certain critical threshold, $\rho_c = \rho_c(\lambda)$, the system fixates, i.e.\ each vertex gets visited finitely many times almost surely. Instead, if the particle density is above the critical threshold, each vertex gets visited infinitely many times almost surely.

Significant efforts have been devoted to establishing various properties of the critical curve $\rho_c = \rho_c(\lambda)$ on various graphs, such as non-triviality \cite{Shellef,  Sidoravicius, Stauffer},
continuity \cite{Taggi3}, universality \cite{Rolla3}, that it is strictly below one for each $\lambda \in (0, \infty)$ \cite{AsselahForienGaudilliere2024, Basu, Forien, HoffmanRichey,  Rolla2,  Stauffer, Taggi1, Taggi2}, and to provide some asymptotic bounds when taking the limits $\lambda \rightarrow 0$ or $\lambda \rightarrow \infty$ \cite{AsselahForienGaudilliere2024,  Asselah, kaufman2025asymptoticbehaviorcriticaldensity}.

In order to observe self-organised critical behaviour, however, one should consider the activated random walk model in a different setting, namely on a large finite subset of a graph under the addition of active particles. More precisely, active particles are added to a random vertex of the set and particles are absorbed at the boundary of this set during relaxation. Every time a new active particle is added, one waits until the system globally stabilises, i.e.\ all particles fall asleep or reach the boundary, where they are removed from the system
(this is called the \emph{wired chain} in the terminology of \cite{LevineSilvestri}; 
related dynamics have also been studied in \cite{Bristiel,LevineLiang} in connection with mixing-time bounds). 
If the average density of particles is too large, then the addition of {several} active particles causes an intense activity, and a huge number of particles is absorbed at the boundary. On the other hand, if the average density of particles is too small, then the addition of an active particle does not activate many other particles, thus mass tends to accumulate. With this carefully designed mechanism, the model is attracted to a stationary state with an average density that should approach $\rho_c$, as the finite volume exhausts the graph, $\rho_c$ being the point at which the infinite system displays critical behaviour. 

The present paper provides a rigorous description of this mechanism for the activated random walk model on the complete graph and provides a characterisation of the support of the invariant distribution of the corresponding Markovian dynamics. We show that the number of particles concentrates with high probability around a value $\rho_c N + a \sqrt{N \log N}$,  {with fluctuations of order at most} $o(\sqrt{N\log N})$, where $N$ is the number of vertices in the underlying graph. We also determine the critical value $\rho_c = \frac{\lambda}{1+\lambda}$ of the particle density on the complete graph, and the lower order shift $a = \sqrt{\lambda}/(1+\lambda)$. As far as we know, together with the cases of totally asymmetric jumps on $\mathbb{Z}$ \cite[Theorem 1]{Cabezas} {and directed jumps on $\mathbb{Z}$ \cite[Theorem 3.2]{RollaSurv}}, this is the only case in which the value of the critical density for the activated random walk model is known \textit{exactly}. In addition, the presence of the lower order shift, that we also calculate, is a feature that should be present for ARW on finite subsets of $\mathbb{Z}^d$ as well. The shift plays an important role in the ARW dynamics, and we believe that the arguments we develop in this paper to study it gives useful insight into a potential spatial analogue in $\mathbb{Z}^d$, at least at a heuristic level. 
We note that an analogous shift in density for the Abelian sandpile on the complete graph and on finite subsets of $\mathbb{Z}^2$ is also present \cite{FeyLevineWilson,JaraiElvidge}. Indeed, an  infinitesimal super-criticality appears to be a general feature of SOC, see the detailed study of the forest-fire model on the complete graph by R\'ath and T\'oth \cite{RathToth}.
We also note that the recent computation of the exact critical density for the stochastic sandpile model on the complete graph \cite{CampaillaForien2025}, which appeared after the publication of our preprint, is closely related in spirit to the present work. 
Moreover, a recent result \cite{JungeKaufmanMeisel2025}, also obtained after the appearance of our preprint, shows that the critical density for activated random walk on $\mathbb{Z}^d$ approaches, as $d\to\infty$, the critical value on the complete graph, namely $\lambda/(1+\lambda)$, which is precisely the value computed in the present paper.

Let us now define the ARW Markov chain formally. We use $\mathcal{G}_{N+1} = (V_{N+1}, E_{N+1})$ to denote the complete graph with $N+1$ vertices, { no self-loops and no multiple edges}. We distinguish a vertex $\partial V_{N+1} \in V_{N+1}$ that plays the role of the \emph{boundary}. We further write $\mathring{\mathcal{G}}_{N+1}=(\mathring V_{N+1}, \mathring E_{N+1})$ for the sub-graph of $\mathcal{G}_{N+1}$ induced by the non-boundary vertices, which has $N$ vertices. The dynamics take place in $\mathcal{G}_{N+1} = (V_{N+1}, E_{N+1})$. A particle configuration can be represented as \smash{$\xi \in \Sigma:=  \{ 0, s \}^{ {\mathring{V}_{N+1}}}$}, where  $\xi(x) =  s$ denotes the presence of a sleeping particle at $x$,
while $\xi(x) = 0$ means that the vertex $x$ is empty.

We initiate the system in an arbitrary configuration of only sleeping particles, $\xi_0 \in \Sigma$. At each time step $n \geq 1$, we add an active particle to a vertex uniformly chosen in $\mathring{V}_{N+1}$ 
(if this vertex already hosts a sleeping particle, the latter is activated) and then let the system evolve according to the continuous time activated random walk dynamics defined above. Whenever an active particle jumps to the boundary vertex, it instantaneously disappears from the system. After an almost surely finite time we obtain a stable configuration, i.e.\ all particles in $ \mathring{V}_{N+1}$ are in the S-state
and some particles may possibly have left the set $ \mathring{V}_{N+1}$ and have disappeared from the system  by jumping to the boundary vertex. We  identify such a stable particle configuration with an element $\xi_{n} \in \Sigma.$

The sequence of random variables $\xi_0, \xi_1, \xi_2, \ldots$ thereby defined is a discrete time Markov chain on the finite configuration space $\Sigma$.
It is straightforward to see, that the chain is irreducible and aperiodic and thus converges to a stationary distribution, $\bar\mu = \bar\mu^{(N)}$,
which does not depend on the initial particle configuration. { Due to the complete symmetry of the model, it suffices to consider the distribution $\mu=\mu^{(N)}$ of the total number of particles under $\bar\mu$. A $\bar\mu$ distributed configuration can be sampled by first drawing sleeping particles according to $\mu$ and then assigning them locations in $\mathring V_{N+1}$ uniformly without replacement.}

Our main results, Theorems \ref{thm:main} and \ref{thm:shift} below, give a detailed characterisation of the support of the stationary distribution $\mu$. Note that although Theorem \ref{thm:shift} supersedes Theorem \ref{thm:main}, we prefer to state both, for two reasons. On the one hand, Theorem~\ref{thm:main} is easier to prove, and on the other hand, the upper bound on the support provided by Theorem \ref{thm:main} is used as input in the proof of the more precise bounds of Theorem \ref{thm:shift}.


\begin{theorem}[Law of large numbers]
	\label{thm:main}
	There exists some sequence $(\alpha_N)_{N\in\mathbb{N}}$ depending only on $\lambda$ and satisfying $\alpha_N=O(\sqrt{N\log N})$, such that,
	$$
	\lim_{N \to \infty} \mu^{(N)} 
	 \Big [ \, \rho_c N- \alpha_N,\rho_c N + \alpha_N \, \Big ]  = 1,
	$$ 
	where 
	$$
	\rho_c =  \frac{\lambda}{1+ \lambda}.
	$$
	
\end{theorem}
In other words, independently of the initial number of particles
at time zero (which may be much larger or smaller than  $\rho_c N$),
the system converges to a critical state in which  the number of particles \textit{concentrates}
around the value 
$\frac{\lambda}{1+ \lambda} N$
{with fluctuations  of order at most} $O( \sqrt{ N \log N}   )$, and thus much smaller than the number of particles.

Our second theorem characterises the lower order shift in the number of sleeping particles more precisely and in particular shows that, at stationarity, the particle density is slightly \emph{above} $\rho_c$.

\begin{theorem}[Lower order shift]
	\label{thm:shift} \ \\
	For any $\eps>0$, we have
	\eqnst
	{ \lim_{N \to \infty} \mu^{(N)} \, \Big [ \rho_c N+ (a-\eps)\sqrt{N\log N}, \rho_c N + (a+\eps)\sqrt{N\log N} \Big ]
		= 1, }
	where
	$$a =\frac{\sqrt{\lambda}}{1+\lambda}.$$  
\end{theorem}

The {lower order shift constant} $ a = \frac{\sqrt{\lambda}}{1+\lambda}$ has in fact an interpretation in terms of moderate deviations of an Ornstein-Uhlenbeck process that appears as scaling limit of appropriately recentered particle counts in our model. 
We discuss this connection  in Section~\ref{ssec:outline}.

\section{The induced chain of particle counts}
In order to characterise the stationary distribution $\mu = \mu^{(N)}$ of the Markov chain $\xi = (\xi_t)_{t \in \mathbb{N}_0}$, we first introduce a new `microscopic' chain $(\eta_t)_{t \in \mathbb{N}_0}$, which has the same stationary distribution as $\xi$. We then argue that $\eta$ can be reduced to a simple bivariate process that keeps only track of particle counts but not of their locations. 

The Markov chain $(\eta_t)_{t \in \mathbb{N}_0}$ is defined as follows: The dynamics take place on $\mathcal{G}_{N+1}$.  Particles are either \emph{asleep (S-particles)} or \emph{active (A-particles)}. The state of the system at time $t$, $t \in \mathbb{N}_{0}$, is denoted by $\smash{\eta_t \in \{0,  s, 1\}^{\mathring V_{N+1}}}$, where $\eta_t(x) =1$ 
(resp. $\eta_t(x) = s$) if there is an active  (resp. sleeping) particle at vertex $x$, while $\eta_t(x)=0$ if there is no particle at $x$. Note that  there is at most one particle per vertex, which might be an S-particle or an A-particle. The single parameter that moderates the active-sleep-transition is the \emph{sleeping rate} $\lambda>0$. The initial configuration is given by $\eta_0=(1,\dots,1)$, i.e.\ there is an active particle at each vertex.

\textit{Update dynamics for $\eta$.} To update the system from $\eta_t$ to $\eta_{t+1}$ we use the following rules: If the configuration $\eta_t$ contains no A-particle, then $\eta_{t+1}=\eta_t$ and we say that \emph{stabilisation} has occurred. If there is at least one A-particle, then we choose a site $x$ occupied by an A-particle uniformly at random. With probability $\smash{\frac{\lambda}{1+\lambda}}$, the particle falls immediately asleep, we set $\eta_{t+1}(x)=s$ and $\eta_{t+1}(y)=\eta_{t}(y), y\neq x$. 
With probability $\smash{\frac{1}{1+\lambda}}$, the chosen particle moves according to a simple random walk on $\mathcal{G}_{N+1}$ until it hits an unoccupied vertex or the boundary. 
If the boundary is hit, the particle is removed from the system.
If an unoccupied site $y$ is hit, then we set $\eta_{t+1}(y)=1$, i.e.\ the particle remains there.  In either case, all occupied sites that were visited by the moving particle and contain an S-particle change their status from $s$ to $1$, i.e.\ the sleeping particles on the path of the chosen active particle are woken up to obtain $\eta_{t+1}$.

\textit{Reduction to particle counts.} We now associate to each realisation of the interacting particle system two variables which count the total number of particles, $X_t$, and the total number of active particles $Y_t$. More precisely, these variables are defined as follows,
$$
X_t := \sum\limits_{x \in \mathring V_{N+1}} \mathbf{1}\{\eta_t(x) \neq 0\},\quad Y_t := \sum\limits_{x \in \mathring V_{N+1}} \mathbf{1}\{\eta_t(x) =1 \},\quad t\in\mathbb{N}_0.
$$
Thus, the total number of S-particles at time $t$ is $X_t - Y_t$ and the total number of empty non-boundary vertices is $N - X_t$.
Recall that due to the complete symmetry of our model, $(X, Y)$ encodes the behaviour of $\eta$ up to a trivial labelling operation: if $\eta'_t$ is obtained from $(X_t,Y_t)$ by uniformly assigning to each particle a site (without replacement), then $\smash{\eta'_t \overset{d}{=}\eta_t}$. All our results about $\mu$ are obtained in the following sections by carefully analysing the behaviour of $(X,Y)$. The crucial connection between $(X,Y)$ and the stationary distribution of $\xi$ is formalised in Lemma~\ref{lem:statdis} below. Namely, if we start with a full system of $N$ active particles and stabilise, then the stabilised state is stationary, as observed also in \cite{LevineLiang}.

Here and throughout the remainder of the article, we denote by $P_N$ and $E_N$ the law and expectation, respectively, of the Markov chain $(X,Y)$ on $\mathcal{G}_{N+1}$ and set $\mathcal{F}_t=\sigma((X_s,Y_s), 0\leq s\leq t)$. 

\begin{lemma}\label{lem:statdis}
	The stationary distribution $\mu=(\mu_k^{(N)})_{0\leq k\leq N}$ of $(\xi_t)_{t \in \mathbb{N}_0}$ has the representation
	\[
	\mu_k =P_N \left [ \, X_{\tau}=k| (X_0,Y_0)=(N,N) \right ], \quad 0\leq k\leq N,
	\]
	where
	\[
	{\tau}=\inf\{t\geq 0:\, Y_t=0\}.
	\]
\end{lemma}
\begin{proof}
	Invariance of $\mu$ follows from the Abelian property { demonstrated, e.g.,\ in \cite{Rolla}}. Consider the configuration in which each vertex in $\mathring{V}_{N+1}$ contains an A-particle and add an extra particle at a uniformly chosen vertex.
	Stabilise in two different ways: first, we can start with stabilising only 
	the extra particle, which will eventually leave the system, as all 
	vertices are occupied by an active particle. Subsequently, we stabilise the $N$ active sites,
	which results in the distribution $\mu$. Second, disregard the extra particle,
	and stabilize the rest, and subsequently stabilise with the extra
	particle present. The resulting distribution corresponds to the distribution $\mu$ evolved by one step of $(\xi_t)_{t \in \mathbb{N}_0}$. Hence $\mu$ is invariant for the 
	one step evolution of $(\xi_t)_{t \in \mathbb{N}_0}$.
\end{proof}

\paragraph{Outline of the paper.}

The remainder of the paper is organized as follows.
In Section~\ref{sect:proofthm1} we introduce a key auxiliary process $S_t$, which measures the deviation of the system from the critical balance between active and sleeping particles, and reduce the proof of Theorem~\ref{thm:main} to three probabilistic estimates, stated as Propositions~\ref{prop:T+bnd}--\ref{prop:exit+}. Proposition~\ref{prop:T+bnd} is proved there, while Propositions~\ref{prop:exit-} and~\ref{prop:exit+} are established later in the paper. Assuming these estimates, Theorem~\ref{thm:main} is proved at the beginning of Section~\ref{sect:proofthm1}.

Propositions~\ref{prop:exit-} and~\ref{prop:exit+} provide bounds on large fluctuations of the process $S_t$ and are proved in Section~\ref{sec:dbounds} via a detailed analysis based on the drift and variance estimates established in Lemma~\ref{lem:moments-new-gen}.
Section~\ref{sec:constant} is devoted to the proof of Lemma~\ref{lem:moments-new-gen} and to refining the estimates on the increments of $S_t$ near the critical region.
Finally, Section~\ref{sec:scaling-window} contains the proof of Theorem~\ref{thm:shift}, where we analyse the dynamics near the critical density and establish the asymptotic lower-order shift.
This section constitutes the main technical part of the paper; see Section~\ref{ssec:outline} for an overview of the proof strategy for Theorem~\ref{thm:shift}.

\section{Proof of Theorem~\ref{thm:main} from deviation bounds}
\label{sect:proofthm1}
Unless stated otherwise, we assume that $P_N [ X_0=N,Y_0=N]=1$ from now on. Our next result shows, that stabilisation occurs after at most $O(N^2)$ steps with very high probability. For any $\rho\in(0,1)$ define the stopping time
\begin{equation}\label{eq:stoppingtime}
\tau_\rho= \inf\{t\geq 0:\, X_t\leq \rho N\}.
\end{equation}
Let further $(\eps_N)_{N\in \mathbb{N}}$ denote a fixed sequence satisfying {$\lim_{N\to\infty}\eps_N=0$ and}
\eqn{e:epsN-assump}
{ \sqrt\frac{{\log N}}{{N}} \leq \varepsilon_N  \leq  \frac{\rho_c}{4}, \quad N\in\mathbb{N}}
and let us fix  $\delta = 1 -  \frac{\rho_c}{2}$. 
\begin{proposition}[Upper bound on stabilisation time]\label{prop:T+bnd}
	There exist constants $c_1,C_1$, depending only on  $\lambda$, such that
	\[
	P_N \left [ {\tau}\wedge \tau_{\rho_c-\varepsilon_N}>\delta(1+\lambda)N^2 \right ]\leq C_1\textup{e}^{-c_1 N}.
	\]
\end{proposition}
\begin{proof}
	Suppose that $X_t > (\rho_c - \varepsilon_N) N$ and $Y_t \ge 1$. The active particle that is
	selected for a potential move can leave the system on its first step with 
	probability $\frac{1}{1+\lambda} \frac{1}{N}$, hence we have
	\[
	P_N \left [ X_{t+1} = X_t - 1 \,|\, \mathcal{F}_t  \right ]
	\ge \frac{1}{1+\lambda} \frac{1}{N}, 
	\]
	as long as $t<{\tau}\wedge \tau_{\rho_c-\varepsilon_N}$. Therefore, the time it takes for either
	$Y_t$ to reach $0$, or $X_t$ to reach $(\rho_c - \varepsilon_N) N$, is 
	stochastically smaller than the time it takes to accumulate 
	$(1 - (\rho_c - \varepsilon_N)) N$ successes in a sequence of independent
	Bernoulli trials with probability of success 
	$\smash{p = \frac{1}{1+\lambda} \frac{1}{N}}$. Since
	\[
	1 - (\rho_c - \varepsilon_N) = 1 - \rho_c + \varepsilon_N 
	<\delta,
	\]
	{for all sufficiently large $N$}, it follows that the probability that more than
	$\delta ( 1 + \lambda ) N^2$ trials are needed decays exponentially in $N$, and the statement follows.
\end{proof}

We measure the deviation of $(X_t,Y_t)_{t\in\mathbb{N}_0}$ from the straight line $(x-y)/(N-y)=\rho_c$ through the quantity
\begin{equation}\label{eq:deviation}
S_t= Y_t-\ell(X_t), \quad t\geq 0,
\end{equation}
with $\ell(x)=\ell_N(x)=(1+\lambda)x-\lambda N$; {see Figure \ref{fig:lx}.}
\begin{figure}
    \centering
    \includegraphics[scale=0.60]{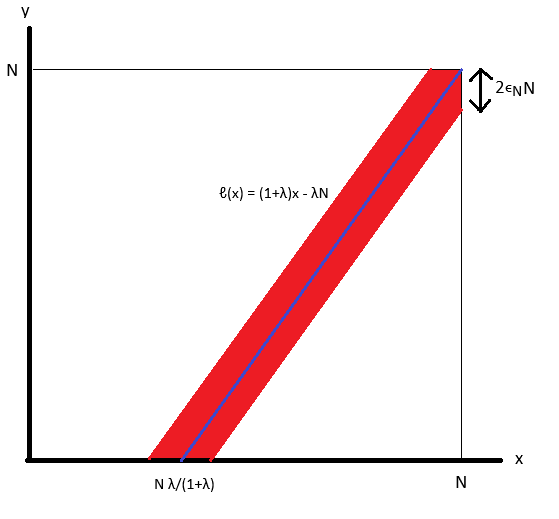}
    \caption{The straight line $y = \ell(x) = (1+\lambda) x - \lambda N$ (in blue) passes through the 
    points $(N,N)$ and $(N \lambda / (1+\lambda), 0) = (\rho_c N, 0)$. The quantity $S_t$ measures
    the vertical distance between $(X_t,Y_t)$ and this line. As we show in 
    Propositions \ref{prop:exit-} and \ref{prop:exit+}, starting in $(N,N)$, the process 
    $(X_t,Y_t)$ stays, with high probability, in the shaded (red) region, and hence hits 
    the horizontal axis near $(\rho_c N, 0)$.}
    \label{fig:lx}
\end{figure}
{As we will prove in Section~\ref{sec:constant}, given $(X_{t-1}, Y_{t-1})$, the expectation of $S_t$ is negative when $Y_{t-1} > \ell(X_{t-1})$ and positive when $Y_{t-1} < \ell(X_{t-1})$. Therefore, the line $\ell(x)$ acts as an equilibrium line for the drift of $S_t$, towards which the process is dynamically attracted.
In contrast, the evolution of the total particle count $X_t$ follows a simpler dynamic.  } {  The tail behaviour of the invariant measure $\mu$ is primarily governed by the fluctuations of the process $(S_t)_{t \geq 0}$ near the stabilisation time. In the next two results, we derive precise bounds on the deviations of $(S_t)_{t \in \mathbb{N}}$, which in turn yield tail estimates for $\mu$.
}

\begin{proposition}[Lower deviations of $(S_t)_{t\geq 0}$]\label{prop:exit-}  
There exist $c_2 = c_2(\lambda) > 0$ and $C_2=C_2(\lambda)<\infty$ such that
	\[
	P_N \left [ S_t \geq  -2 \varepsilon_N N \text{ for all } 
    0 \le t \le {\tau \wedge} \delta (1 + \lambda) N^2 \right ]
	\ge 1- C_2 \, N^4 e^{ - c_2 \varepsilon_N^2 N }.
	\]
\end{proposition}

{
\begin{proposition}[Upper deviations of $(S_t)_{t\geq 0}$
	]\label{prop:exit+}  
	 There exist $c_3 = c_3(\lambda) > 0$ and $C_3=C_3(\lambda)<\infty$ such that
	\[
    P_N \left [ S_t \leq  2 \varepsilon_N N \text{ for all } 0 \le t \le \tau \wedge \delta (1 + \lambda) N^2 \right ]
	\ge 1- C_3 \, N^4 e^{ - c_2 \varepsilon_N^2 N }.
	\]
\end{proposition}
}

{
We discuss the derivation of Proposition~\ref{prop:exit-} in Section~\ref{sec:dbounds}. The proof of Proposition~\ref{prop:exit+} is very similar and superseded by the results of Section~\ref{sec:scaling-window}. We state Proposition~\ref{prop:exit+} here to allow for a  proof of Theorem~\ref{thm:main}.}

\begin{proof}[Proof of Theorem~\ref{thm:main}]

Recall that $\delta=1-\rho_c/2$. Define
\[
\bar\varepsilon_N:= c \sqrt{\frac{\log N}{N}},
\qquad
\varepsilon_N:=\frac{1+\lambda}{2}\bar\varepsilon_N,
\]
where $c = c(\lambda)$ is a large enough constant,
so that both sequences satisfy \eqref{e:epsN-assump} for all sufficiently large $N$.
Define
\[
\alpha_N:=\bar\varepsilon_N N,
\]
so that $\alpha_N=O(\sqrt{N\log N})$.
Set
\[
T:=\delta(1+\lambda)N^2.
\]

\noindent 

Recall the definition of $S_t$ in  (\ref{eq:deviation}).
For the upper tail bound,  by using Lemma~\ref{lem:statdis} in the first step and the definition of $S_t$  afterwards we have 
\begin{align*}
\mu^{(N)}\big ( \rho_cN+\alpha_N,\infty \big ) & = P_N \left [ X_\tau>\rho_cN+\alpha_N \right ] \\
&\le
P_N  \left [S_\tau<-2\varepsilon_NN,\ \tau\le T \right ]
+
P_N   \left [ \tau>T,\ X_\tau>\rho_cN+\alpha_N \right ] \\
&\le
P_N\Bigl [ \exists\, t\le \tau\wedge T:\ S_t<-2\varepsilon_N N \Bigr ]
+
P_N  \left [ \tau>T,\ X_\tau>\rho_cN+\alpha_N \right ].
\end{align*}
The first term is $o(1)$ by Proposition~\ref{prop:exit-}.
For the second term,  monotonicity of $(X_t)_{t\ge 0}$ gives
\begin{align*}
P_N \big [   \tau > T, X_\tau > \rho_c N + \alpha_N  \big ]  & = 
P_N \big [ X_\tau > (\rho_c  + \bar \varepsilon_N) N, \tau > T    \big ] \\
& \leq P_N \big [  X_\tau > (\rho_c  - \bar \varepsilon_N) N, \tau > T    \big ] \\
& \leq  P_N \big [ \tau_{\rho_c -\bar \varepsilon_N} > T,  \tau > T    \big ] = o(1),
\end{align*} 
where in the last step we applied Proposition~\ref{prop:T+bnd}.

For the lower tail bound, again using Lemma~\ref{lem:statdis} in the first step we have 
\begin{align*}
&\mu^{(N)} \big [0,\rho_cN-\alpha_N \big ) 
= P_N [X_\tau<\rho_cN-\alpha_N ]  \\
&\quad \le 
P_N [ X_\tau<\rho_cN-\alpha_N,\ \tau\le T ]
+ P_N [ \tau_{\rho_c-\bar\varepsilon_N}\le T<\tau ] 
+ P_N [\tau\wedge\tau_{\rho_c-\bar\varepsilon_N}>T].
\end{align*}
The third term is $o(1)$ by Proposition~\ref{prop:T+bnd}.

For the first term, if $X_\tau<\rho_cN-\alpha_N$ and $\tau\le T$, then $S_\tau>2\varepsilon_NN$, hence
\[
P_N 
\left [ X_\tau<\rho_cN-\alpha_N,\ \tau\le T \right ]
\leq 
P_M \left [  \exists\, t\le \tau\wedge T:\ S_t>2\varepsilon_NN \right ] = o(1),
\]
where the last step follows from Proposition~\ref{prop:exit+}.
For the second term, on $\{\tau_{\rho_c-\bar\varepsilon_N}\le T<\tau\}$ we have
$Y_{\tau_{\rho_c-\bar\varepsilon_N}}>0$ and
$X_{\tau_{\rho_c-\bar\varepsilon_N}}\le (\rho_c-\bar\varepsilon_N)N$, so
by definition of $S_t$ we have 
$
S_{\tau_{\rho_c-\bar\varepsilon_N}}
>
2\varepsilon_NN.
$
Hence,
\[
P_N \big [ \tau_{\rho_c-\bar\varepsilon_N}\le T<\tau \big ]
\leq P_N \big [ \exists\, t\le \tau\wedge T:\ S_t>2\varepsilon_NN \big ]  = o(1),
\]
where the last step follows from  Proposition~\ref{prop:exit+}.
Therefore
\[
\mu^{(N)}\big  [0,\rho_cN-\alpha_N \big )=o(1).
\]
The desired bound then follows from the combination of the  the upper and lower tail bounds.

\end{proof}

\section{Transition probabilities for the particle count chain}\label{sec:tP}
Our main technical tool for the remainder of the paper are the following formulae for the transition probabilities of $(X,Y)$. For any (random or deterministic) sequence $(a_n, n\in\mathbb{Z})$ we use the notation
\[
\Delta a_n = a_{n+1}-a_n,\quad n\in\mathbb{Z},
\]
for the forward difference operator. { In the case where the $a_n$ themselves are functions of further variables, $\Delta$-differentiation is always understood to apply to the right sub-index variable $n$.}
We define for each integer $k \geq 0$ and $j \in \{0, -1\}$
$$
	\Pi_k^{(j)}(x,y)  := P_N \Big [ \Delta Y_{t} \ge k + j, \, \Delta X_t=j \, \big | \, X_t=x,Y_t=y \Big  ].
$$
\begin{lemma}\label{lemma:exactcomputation}
	For any $0\leq y\leq x \leq N$, we have for  $1\leq k\leq x-y$ 
	\begin{align*}
		\Pi_k^{(0)}(x,y) & = 
		\frac{1}{1+\lambda}   \,\frac{N+1}{N}  \, 
		\prod_{i=0}^{k-1}   \Big (
		\frac{x - y - i}{N + 2  - (y + i)} \Big )  \, 
		\frac{N +1 - x}{N+2 - x},  
	\end{align*} 
and for  $0 \leq k\leq x-y$ 
	\begin{align*}		
		\Pi_k^{(-1)}(x,y) = 
		\frac{1}{1+\lambda}   \,\frac{N+1}{N}  \, 
		\prod_{i=0}^{(k-1)_+}   \Big (
		{\frac{x - y - i}{N + 2  - (y + i)}} \Big )  \, 
		{\frac{1}{N+2 - x}}.
	\end{align*}
	Moreover, we have 
	$$
		\Pi_0^{(-1)}(x,y) = 
		\frac{1}{1+\lambda}   \frac{N+1}{N} \,  \,   \frac{1}{N   + 2 - x}.  
	$$
\end{lemma}
\begin{proof}
	We call a vertex hosting one (or more during updates) active particle an `A-vertex', a vertex hosting an S-particle an `S-vertex', and a vertex hosting at least one particle a `filled vertex'. The `sleeping set' is the set of S-vertices in $\mathring{V}_{N+1}$ (at a given time $t$).
	
	\textit{Derivation of the formula for $\Pi_k^{(0)}(x,y)$.}
	An active particle is drawn uniformly at random. We move the particle until it either reaches one empty vertex or it reaches the boundary $\partial V_{N+1}$.	The event defining $\Pi_k^{0}(x,y)$ occurs if and only if the A-particle visits $k$ distinct sites of the sleeping set at least once and, after that, reaches an empty site rather than reaching the boundary. We now compute the probability of this event.
	\begin{enumerate}[label=(\arabic*)]
	\setlength{\itemsep}{0pt}
\setlength{\topsep}{0pt}
		\item \textit{Probability that the moving particle hits one vertex of the sleeping set.} First, the probability that the A-particle reaches one vertex {with a sleeping particle} before reaching the boundary or an empty vertex is 
		\begin{equation}\label{eq:case1.3}
			\frac{1}{1+\lambda} \, \,  \, \, \frac{X_t - Y_t}{N + 2 - Y_t} \frac{N+1}{N}.
		\end{equation}
		This expression is the sum of two terms: The first term is the probability that the moving particle jumps immediately to the sleeping set,
		\begin{equation}\label{eq:case1.1}
			\frac{1}{1+\lambda} \, \, \frac{X_t - Y_t}{N},
		\end{equation}
		the second term is the probability that it first jumps to one or more A-vertices and, after that, it jumps to the sleeping set rather than reaching the boundary or an empty vertex
		\begin{equation}\label{eq:case1.2}
			\frac{1}{1+\lambda} \, \, \frac{Y_t-1}{N} \, \, \frac{X_t - Y_t}{N  + 1 - (Y_t - 1)} = 
			\frac{1}{1+\lambda} \, \, \frac{Y_t-1}{N} \, \, \frac{X_t - Y_t}{N  + 2 - Y_t }.
		\end{equation}
		Here, the product of the first two factors is the probability that at the first step the particle jumps to a vertex hosting an A-particle and the third factor is the probability that, after having jumped over an arbitrary number of vertices hosting an A-particle, it jumps to the sleeping set rather than reaching the boundary or an empty site.
		The sum of (\ref{eq:case1.1}) and (\ref{eq:case1.2}) gives (\ref{eq:case1.3}).
		\item \textit{Probability that moving A-particle hits a second vertex of the sleeping set.}	At this point the A-particle is on a vertex which was hosting an S-particle in $\eta_{t}$ and which now hosts two active particles, thus  $Y_t$ vertices host at least one active particle (one of them hosts two active particles) and $X_t - Y_t - 1$ vertices host one sleeping particle. The probability that the A-particle will leave the set of vertices hosting one active particle by jumping to a vertex hosting one sleeping particle rather than reaching the boundary or an empty site is
		\begin{equation}\label{eq:case1.4}
			\frac{X_t - Y_t - 1}{N + 1 - Y_t},
		\end{equation}
		where the numerator is the current number of vertices hosting a sleeping particle and the denominator is the number of vertices which host one sleeping particle, are empty or the boundary vertex.
		\item\textit{Probability that the A-particle hits a third vertex of the sleeping set.} Now, the active particle
		is on a vertex which was hosting an S-particle in $\eta_{t}$ and which now hosts two active particles, thus  $Y_t+1$ vertices host at least one active particle (one of them hosts two active particles) and $X_t - Y_t -2$ vertices host one sleeping particle. Similarly to (\ref{eq:case1.4}), the 
		probability that the A-particle will leave the set of vertices hosting one active particle by jumping to a vertex hosting one sleeping particle rather than reaching the boundary or an empty vertex is,
		\begin{equation}\label{eq:case1.5}
			\frac{X_t - Y_t - 2}{N + 1  - (Y_t + 1)}.
		\end{equation}
		\item\textit{Iteration.} Iteratively, we obtain that
		\begin{multline}\label{eq:case1.6}
			P_N \Big [ Y_{t+1} > Y_{t} + k, \, \,  X_{t+1}= X_t \, \big | \,  \mathcal{F}_t \Big ] \\ =  \,   
			\frac{1}{1+\lambda} \, \, \frac{N+1}{N}  \, \, \frac{X_t - Y_t}{N+2 - Y_t} \, \,
			\prod_{i=1}^{k}   \Big (
			\frac{X_t - Y_t - i}{N + 2  - (Y_t + i)} \Big ) \, \, 
			\frac{N +1 - X_t}{N+2 - X_t},
		\end{multline}
		where the first three factors are given by  (\ref{eq:case1.3}), the product of $k$ terms corresponding to the third factor is obtained iterating (\ref{eq:case1.4}) and (\ref{eq:case1.5}) $k$ times, and the last factor is the probability that the last jump is to an empty vertex rather than the boundary. This concludes the calculation of $\Pi_k^{(0)}(x,y)$.
	\end{enumerate}
	
	\textit{Derivation of the formula for $\Pi_k^{(-1)}(x,y)$.}
	The event defining  {$\Pi_k^{(-1)}(x,y)$}  occurs if and only if the same sequence of events as in the previous consideration occurs, with the difference that the last jump is to $\partial V_{N+1}$ rather than to an empty vertex. We immediately obtain the stated formula for $\Pi_k^{(-1)}(x,y)$. \\
	
{\textit{Derivation of the formula for $\Pi_0^{(-1)}(x,y)$.}
	Observe that in this case $\Pi_0^{(-1)}(x,y) =   P_N  [ \Delta Y_{t} \ge -1, \, \Delta X_t= -1 \,  | \, X_t=x,Y_t=y  ] =  P_N  [  \Delta X_t= -1 \, \big | \, X_t=x,Y_t=y  ] $.
	The event  $\Delta X_t= -1 $ occurs if and only if one of the following two situations happens: 
	{(a)}: The active particle jumps immediately to the boundary vertex, 
	which occurs with probability $\frac{1}{1 + \lambda} \frac{1}{N}$. 
		{(b)}:  The active particle jumps immediately to one of the vertices hosting at least one particle,  and, after jumping an arbitrary number of times on this set, it leaves such a set by jumping directly to the boundary vertex rather than to an empty site. This event occurs with probability 
		$$
		\frac{1}{1 + \lambda} \frac{X_t - 1}{N} \, \, \frac{  1  }{ N - X_t + 2  }
		$$
		The sum of the two probabilities leads to the result.}
\end{proof}

\section{Proof of Propositions \ref{prop:exit-} and \ref{prop:exit+}}
\label{sec:dbounds}

{
We now examine the evolution of $(S_t)_{t\geq 0}$ with the goal to prove 
Propositions~\ref{prop:exit-} and \ref{prop:exit+}. This will be achieved with
the help of Lemma \ref{lem:moments-new-gen}, that we state in Section \ref{ssec:moments},
and prove in Section \ref{sec:constant}.
}

{We use the notation
\[
\pi^{(j)}_k(x,y) = -\Delta {\Pi_{k}^{(j)}}(x,y),\quad {0}\leq k\leq x-y, j\in\{-1,0\},
\]
for the increment distribution of $(X,Y)$ in the proofs of this section. Note that, whenever $x \geq y$, 
and ${0}\leq k\leq x-y$,  we have
$$
\pi^{(j)}_{k}(x,y) = P_N \Big [  \Delta Y_{t} =  k + j, \, \Delta X_t=j \, \big | \, X_t=x,Y_t=y \Big ].
$$
Moreover we define for  $x \geq  y \geq 1$
\[
\pi^{(0)}_{-1}(x,y) =   P_N \Big [ \Delta Y_{t} =  -1, \, \Delta X_t=0 \, \big | \, X_t=x,Y_t=y \Big ] =  \frac{\lambda}{1+\lambda}.
\]}
{
Our goal is to show that, if $S_t$ is {far below, respectively, far above $0$}, then it has a considerable drift upwards, respectively, downwards in the next step, resulting in reduction, respectively, increase, of the number of S-particles with high probability. 
{Define the events
$$L^{(N)-}=\{-2 \eps_N N \le S_t \le -\eps_N N\},$$
and
$$L^{(N)+}=\{\eps_N N \le S_t \le 2 \eps_N N\}.$$
}
}

\subsection{Conditional mean and variance of the increments of $S_t$}
\label{ssec:moments}

{
\begin{lemma}
\label{lem:moments-new-gen}
(i) There exists $N_0$ such that on the event $L^{(N)-}$, we have 
\eqn{e:drift-rough-gen}
{ \E [ \Delta S_t \,|\, \cF_t ]
  \ge -\frac{S_t}{2(1+\lambda)N}  \quad \text{ for $N \ge N_0$.} }
(ii) There exists $N_0$ such that on the event $L^{(N)+}$, we have 
\eqn{e:drift-rough-gen2}
{ \E [ \Delta S_t \,|\, \cF_t ]
  \le -\frac{S_t}{2(1+\lambda)N}  \quad \text{ for $N \ge N_0$.} }
(iii) There exists $C = C(\lambda)$ such that we have
\eqn{e:var-rough-gen}
{ \E [ (\Delta S_t)^2 \,|\, \cF_t ]
  \le C, \quad \text{ for all $N \in \mathbb{N}$ and all $t \ge 0$.} }
(iv) On the event $L^{(N)+}$, we have 
\eqn{e:2nd-rough-gen}
{ \E [ (\Delta S_t)^2 e^{\eps_N \Delta S_t} \,|\, X_t, Y_t ]
  = O(1). }
\end{lemma} 

We prove Lemma \ref{lem:moments-new-gen} in Section \ref{sec:constant}. We now proceed to
prove the two propositions, assuming this lemma.
}

\subsection{Bound on negative deviations of $S_t$}
\label{ssec:St-}

{
\begin{lemma}
\label{lem:exp-moment-bnd-}
There exist constants $h>0, N_0\in \mathbb{N}$ depending only on $\lambda$ and $(\varepsilon_N)_{N\in\mathbb N}$ such that for every $N>N_0$ on the event $L^{(N)-}$ we have
\eqnst
{ E_N \big[ \exp \big( - h \eps_N \Delta S_{t} \big) \,\big|\, \mathcal{F}_t \big]\le 1 , 
      \quad  t\in\mathbb{N}_0. }
\end{lemma}

\begin{proof}
We use the inequality
\[ \exp \big( - h \eps_N (S_{t+1} - S_t) \big)
   \le 1 - h \eps_N (S_{t+1} - S_t) + h^2 \eps_N^2 (S_{t+1} - S_t)^2, \]
which holds when $h \eps_N$ is sufficiently small, because $S_{t+1} - S_t \ge -1$ always.
Taking the conditional expectation given $\cF_t$, and using Lemma \ref{lem:moments-new-gen}(i),(iii),  
we get that the right hand side is at most
\[ 1 + h \eps_N \frac{S_t}{2 (1+\lambda) N} + O ( h^2 \eps_N^2 ). \]
On the event $L^{(N)-}$, we bound $S_t \le - \eps_N N$, resulting in the upper bound
\[ 1 - \frac{1}{2(1+\lambda)} h \eps_N^2 + O ( h^2 \eps_N^2 ). \]
This is indeed less than $1$, if $h > 0$ is set sufficiently small, and $N$ is large enough.
\end{proof}

We are now ready to prove Proposition~\ref{prop:exit-}.
}

\begin{proof}[Proof of Proposition~\ref{prop:exit-}]
For $1 \le s < t$, we introduce the following event:
\eqnst
{ A_{s,t}
= \left\{ \parbox{9cm}{$-\eps_N N - 1 \le S_s \le -\eps_N N$, and 
for all $s < r < t$ we have $-2 \eps_N N \le S_r \le -\eps_N N$
and $S_t < -2 \eps_N N$} \right\}. }
Since the downward steps of $Y_t$ are at most of size $1$, we have 
that on the event in the statement of the proposition, there 
must exist $1 \le s < t \le \delta (1+\lambda) N^2$ such that 
$A_{s,t}$ occurs. This implies that 
\eqnst
{ P_N \left[ \text{$S_t < -2 \eps_N N$ for some 
$0 \le t \le \delta(1+\lambda) N^2$} \right]
\le \sum_{1 \le s < t \le \delta(1+\lambda) N^2}
\P \left(A_{s,t} \right). }

The number of terms is $O(N^4)$, hence it is enough to upper bound 
each summand by $\exp( - c \eps_N^2 N)$. Let us fix $h > 0$ satisfying the
conclusion of Lemma \ref{lem:exp-moment-bnd-}. Write $h_N = h \eps_N$.
Using Markov's inequality in the
second step, we have
\eqnsplst
{ \P [ A_{s,t} ] 
&\le \P [ S_t - S_s \le - \eps_N N + 1,\, 
- 2 \eps_N N \le S_r \le -\eps_N N \text{ for all $s \le r < t$} ] \\
&= \P \big[ e^{-h_N (S_t - S_s)} \ge e^{-h_N} \, e^{h_N \eps_N N},\, 
- 2 \eps_N N \le S_r \le -\eps_N N \text{ for all $s \le r < t$} \big] \\
&\le C \, \exp (- h_N \eps_N N ) \, \E \big[ e^{-h_N (S_t - S_s)} \, 
\bfone\{- 2 \eps_N N \le S_r \le -\eps_N N \text{ for all $s \le r < t$}\} \big] \\
&= C \, \exp (- h_N \eps_N N ) \, \E \Big[ \prod_{r=s}^{t-1} \E \big[ e^{-h_N (S_{r+1} - S_r)} \,
\bfone\{- 2 \eps_N N \le S_r \le -\eps_N N\} \,\big|\, \mathcal{F}_r \big] \Big] \\
&\le C \, \exp (- h_N \eps_N N ) 
\le C \exp ( - h \eps_N^2 N ). } 
This completes the proof.
\end{proof}

\subsection{Bound on positive deviations of $S_t$}
\label{ssec:St+}

{
\begin{lemma}
\label{lem:exp-moment-bnd+}
{There exist constants $h>0, N_0\in \mathbb{N}$ depending only on $\lambda$ and $(\varepsilon_N)_{N\in\mathbb N}$ such that for every $N>N_0$
on the event $L^{(N)+}$ we have}
\eqnst
{ E_N \big[ \exp \big( h \eps_N \Delta S_{t} \big) \,\big|\, \mathcal{F}_t \big]\le 1 , \quad  t\in\mathbb{N}_0.
	}
\end{lemma}

\begin{proof}
Using Taylor's theorem with remainder to expand the exponential, there exists $\theta \in [0,1]$ such that
\eqnst
{ E_N \big[ \exp \big( h \eps_N \Delta S_{t} \big) \,\big|\, \mathcal{F}_t \big]
  = 1 + h \eps_N E_N \big[ \Delta S_t \,\big|\, \mathcal{F}_t \big]
    + h^2 \eps_N^2 E_N \big[ (\Delta S_t)^2 \exp \big( \theta h \eps_N \Delta S_t \big) 
      \,\big|\, \mathcal{F}_t \big]. }
Using Lemma \ref{lem:moments-new-gen}(ii),(iv)  
we get that the right hand side is at most
\[ 1 - h \eps_N \frac{S_t}{2 (1+\lambda) N} + O ( h^2 \eps_N^2 ). \]
On the event $L^{(N)+}$, we bound $S_t \ge \eps_N N$, resulting in the upper bound
\[ 1 - \frac{1}{2 (1+\lambda)} h \eps_N^2 + O ( h^2 \eps_N^2 ). \]
This is indeed less than $1$, if $h > 0$ is set sufficiently small, and $N$ is large enough.
\end{proof}

We are almost ready to prove Proposition \ref{prop:exit+}, with an argument similar to the one for
Proposition \ref{prop:exit-}. The only extra technicality is that the upwards jumps of $S_t$ do
not remain under a fixed bound. This issue will be dealt with by the following simple bound.
Define the event 
\[ D
   = \left\{ \parbox{12cm}{$\exists$ $0 \le t \le \delta (1+\lambda) N^2$ such that
      $X_t \le N (1 - \frac{\eps_N}{2 (1+\lambda)})$ and 
      $S_{t+1} - S_t > \frac{1}{2} \eps_N N$} \right\}. \]
      
\begin{lemma}
\label{lem:jump-St}
There exist $C = C(\lambda)$ and $c = c(\lambda)$ such that
$\P [ D ] \le C N^2 \exp ( - c \eps_N^2 N )$
\end{lemma}

\begin{proof}
For all $t \in \mathbb{N}_0$, on the event $X_t \le N (1 - \frac{\eps_N}{2 (1+\lambda)})$, we have
\[ \frac{X_t - Y_t - i}{N+2 - Y_t - i}
   \le \frac{X_t}{N}
   \le \left(1 - \frac{\eps_N}{2 (1+\lambda)}\right)
   \le \exp \bigg( - \frac{\eps_N}{2 (1+\lambda)} \bigg). \]
Substituting this into the formulas from Lemma \ref{lemma:exactcomputation}, we get
\[ \P \bigg[ Y_{t+1} - Y_t > \frac{1}{2(1+\lambda)} \eps_N N \,\bigg|\, \cF_t \bigg] 
   \le \exp \bigg( - \frac{\eps^2_N N}{4 (1+\lambda)^2} \bigg). \]
Since $\{ S_{t+1} - S_t > \frac{1}{2} \eps_N N \} \subset \{ Y_{t+1} - Y_t > \frac{1}{2(1+\lambda)} \eps_N N \}$
by the definition of $S_t$, we get
\[ \P \bigg[ \text{$X_t \le N \left(1 - \frac{\eps_N}{2 (1+\lambda)}\right)$ and 
      $S_{t+1} - S_t > \frac{1}{2} \eps_N N$} \Bigg]
      \le \exp \big( - c \eps^2_N N  \big). \]
A union bound over $t$ completes the proof.
\end{proof}

\begin{proof}[Proof of Proposition~\ref{prop:exit+}.]
For $0 \le s < t$, we introduce the following event:
\eqnst
{ B_{s,t}
= \left\{ \parbox{9cm}{$\eps_N N \le S_{s+1} \le \frac{3}{2} \eps_N N$, 
  for all $s+1 < r < t$ we have $\eps_N N \le S_r \le 2 \eps_N N$ and $S_t > 2 \eps_N N$} \right\}. }
Observe that as long as $X_s > N(1 - \frac{\eps_N}{2(1+\lambda)})$ it is impossible to have 
$S_{s+1} > \frac{3}{2} \eps_N N$. 
Therefore, on the event in the statement of the proposition, and when the event $D^c$ also occurs, there 
must exist $0 \le s < t \le \delta (1+\lambda) N^2$ such that the event
\[ D^c \cap \left\{ \parbox{9cm}{$S_s < \eps_N N$, $\eps_N N \le S_{s+1} \le \frac{3}{2} \eps_N N$, 
  for all $s+1 < r < t$ we have $\eps_N N \le S_r \le 2 \eps_N N$ and $S_t > 2 \eps_N N$} \right\} 
  \subset B_{s,t} \]
occurs. This implies that 
\eqnst
{ P_N \left[ \text{$S_t > 2 \eps_N N$ for some 
$0 \le t \le \tau \delta (1+\lambda) N^2$} \right]
\le \P [D] + \sum_{0 \le s < t \le \delta(1+\lambda) N^2}
\P \left(B_{s,t} \right). }
The term $\P [D]$ satisfies the required bound due to Lemma \ref{lem:jump-St}.
The number of terms in the sum is $O(N^4)$, hence it is enough to upper bound 
each summand by $\exp( - c \eps_N^2 N)$. Let us fix $h > 0$ satisfying the
conclusion of Lemma \ref{lem:exp-moment-bnd+}. Write $h_N = h \eps_N$.
Using Markov's inequality in the
second step, we have
\eqnsplst
{ &\P [ B_{s,t} ] 
\le \P [ S_t - S_{s+1} \ge \frac{1}{2} \eps_N N,\, 
 \eps_N N \le S_r \le 2 \eps_N N \text{ for all $s+1 \le r < t$} ] \\
&\quad = \P \big[ e^{h_N (S_t - S_{s+1})} \ge e^{h_N \eps_N N / 2},\, 
 \eps_N N \le S_r \le 2 \eps_N N \text{ for all $s+1 \le r < t$} \big] \\
&\quad \le C \, \exp (- h_N \eps_N N / 2) \, \E \big[ e^{h_N (S_t - S_{s+1})} \, 
\bfone\{\eps_N N \le S_r \le 2 \eps_N N \text{ for all $s+1 \le r < t$}\} \big] \\
&\quad = C \, \exp (- h_N \eps_N N / 2) \, \E \Big[ \prod_{r=s+1}^{t-1} \E \big[ e^{h_N (S_{r+1} - S_r)} \,
\bfone\{\eps_N N \le S_r \le 2 \eps_N N\} \,\big|\, \mathcal{F}_r \big] \Big] \\
&\quad \le C \, \exp (- h_N \eps_N N / 2) 
\le C \exp ( - (h/2) \eps_N^2 N ). } 
This completes the proof.
\end{proof}
}

\section{Proof of conditional moment bounds}
\label{sec:constant}
{
\noindent 
In this section we prove Lemma \ref{lem:moments-new-gen} stated in Section \ref{ssec:moments}. 
We also state and prove here more refined estimates on the increments of $S_t$ when 
$X_t$ is 'near the critical region', that is, when
$X_t = \frac{\lambda}{1+\lambda} N + b (1+o(1)) \sqrt{N \log N}$ for some $b > 0$.
On the one hand, these estimates are needed in Section \ref{sec:scaling-window} in the 
proof of our main result, Theorem \ref{thm:shift}.
On the other hand, they provide a heuristic understanding of the numerical value of the 
constant $a$ in Theorem \ref{thm:shift} through a connection with the Ornstein–Uhlenbeck process.

{Recall from (\ref{eq:deviation})  the definition of the deviation $S_t$ and note that   
\eqn{e:St-form}
{  S_t 
= Y_t -(1+\lambda) X_t + \lambda N
= (Y_t - X_t) + \lambda (N - X_t), \quad t\in \mathbb{N}.}
Also recall that $\Delta S_t = S_{t+1} - S_t$.}


The estimates will rely on the following result, whose (elementary) proof is 
deferred to Section \ref{ssec:proof-sums} in the Appendix.
}

\begin{lemma}
\label{lem:sums}
For any $n \in\mathbb{N}$ and $m > n$ we have
\eqn{e:sum-1st}
{ \sum_{\ell = 1}^n \frac{n (n-1) \dots (n-\ell+1)}{m (m-1) \dots (m-\ell+1)}
= \frac{m+1}{m-n+1} - 1 
= \frac{n}{m-n+1}, }
\eqn{e:sum-2nd}
{ \sum_{\ell = 1}^n \ell \frac{n (n-1) \dots (n-\ell+1)}{m (m-1) \dots (m-\ell+1)}
= \frac{n (m+1)}{(m-n+1) (m-n+2)}.}
\end{lemma}

{
We first compute a general formula for the conditional expectation of the increment of $S_t$.

\begin{lemma}
\label{lem:drift-formula}
For all $t \ge 0$ we have
\eqnspl{e:drift-formula}
{ \E [ S_{t+1} - S_t \,|\, \cF_t ]
  &= \frac{1}{1+\lambda} \frac{N+1}{N} \bigg[ \frac{-S_t}{N+3-X_t} + \frac{-2\lambda}{N+2-X_t} \\
  &\qquad\qquad + \frac{3\lambda}{(N+2-X_t)(N+3-X_t)} + O(N^{-1}) \bigg]. }
\end{lemma}

\begin{proof}
We compute
\eqnspl{e:drift-gen-form}
{ &\E [ S_{t+1} - S_t \,|\, \cF_t ]
= -\pi^{(0)}_{-1} + \sum_{k=0}^{X_t-Y_t} (\lambda+k) \pi^{(-1)}_k + \sum_{k=0}^{X_t-Y_t} k \pi^{(0)}_k \\
&\quad = -\frac{\lambda}{1+\lambda} + \lambda \Pi^{(-1)}_0 
+ \sum_{k=1}^{X_t-Y_t} \sum_{\ell=1}^k (\pi^{(-1)}_k + \pi^{(0)}_k) \\
&\quad = -\frac{\lambda}{1+\lambda} + \lambda \Pi^{(-1)}_0 
+ \sum_{\ell=1}^{X_t-Y_t} \sum_{k=\ell}^{X_t-Y_t} (\pi^{(-1)}_k + \pi^{(0)}_k) \\
&\quad = -\frac{\lambda}{1+\lambda} + \lambda \Pi^{(-1)}_0 
+ \sum_{\ell=1}^{X_t-Y_t} (\Pi^{(-1)}_\ell + \Pi^{(0)}_\ell). }
For the sum over $\ell$ we use Lemma \ref{lemma:exactcomputation} and the identity \eqref{e:sum-1st}, 
which give
\eqnspl{e:ell-sum}
{ \sum_{\ell=1}^{X_t-Y_t} (\Pi^{(-1)}_\ell + \Pi^{(0)}_\ell)
&= \frac{1}{1+\lambda} \frac{N+1}{N} \sum_{\ell=1}^{X_t-Y_t} \prod_{i=0}^{\ell-1}
\frac{X_t-Y_t-i}{N+2-Y_t-i} \\
&= \frac{1}{1+\lambda} \frac{N+1}{N} \frac{X_t-Y_t}{N+3-X_t}. }  
We also have 
\eqnst
{ \Pi^{(-1)}_0
= \frac{1}{1+\lambda} \frac{N+1}{N} \frac{1}{N+2-X_t}. }
Substituting these into \eqref{e:drift-gen-form} and then using \eqref{e:St-form}, we have
\eqnspl{e:drift-gen-form-2}
{ &\E [ S_{t+1} - S_t \,|\, X_t, Y_t ]
= -\frac{\lambda}{1+\lambda} + \frac{\lambda}{1+\lambda} \frac{N+1}{N} \frac{1}{N+2-X_t}
+ \frac{1}{1+\lambda} \frac{N+1}{N} \frac{X_t-Y_t}{N+3-X_t} \\
&\quad = -\frac{\lambda}{1+\lambda} + \frac{\lambda}{1+\lambda} \frac{N+1}{N} \frac{1}{N+2-X_t}
+ \frac{1}{1+\lambda} \frac{N+1}{N} \frac{-S_t + \lambda(N-X_t)}{N+3-X_t}. }
The term containing $-S_t$ equals
\[ \frac{1}{1+\lambda} \frac{N+1}{N} \frac{-S_t}{N+3-X_t}, \]
agreeing with the first term inside the square brackets in \eqref{e:drift-formula}.
The sum of the remaining terms in \eqref{e:drift-gen-form-2} equals
\eqnsplst{
 &\frac{\lambda}{1+\lambda} \frac{N+1}{N} \left[ -1 + O(N^{-1}) + \frac{1}{N+2-X_t}
    + \frac{N-X_t}{N+3-X_t} \right]. }
A short calculation shows that this yields the remaining terms inside the square brackets 
in \eqref{e:drift-formula}.
\end{proof}

\begin{proof}[Proof of Lemma \ref{lem:moments-new-gen}(i).]
Since $-\eps_N N \to -\infty$, on the event $L^{(N)-}$, 
for sufficiently large $N$, we have
\[ \frac{-2\lambda}{N+2-X_t} + O(N^{-1})
   \ge \frac{S_t}{2(N+3-X_t)}. \]
Substituting this into \eqref{e:drift-formula}, and also using that the remaining term is positive, we 
obtain that
\[ \E [ S_{t+1} - S_t \,|\, \cF_t ]
  \ge \frac{-S_t}{2(1+\lambda)} \frac{N+1}{N} \frac{1}{N+3-X_t}
  \ge \frac{-S_t}{2(1+\lambda)N}, \]
completing the proof.
\end{proof}

\begin{proof}[Proof of Lemma \ref{lem:moments-new-gen}(ii).]
The second term inside the square brackets in \eqref{e:drift-formula} is negative, so we may ignore it. 
On the event $L^{(N)+}$, and taking into account \eqref{e:St-form},
we have $\eps_N N \le S_t \le \lambda (N - X_t)$. This gives that the third term inside the 
square brackets in \eqref{e:drift-formula} is at most $3 \lambda^3 \eps_N^{-2} N^{-2}$.
Due to the assumption \eqref{e:epsN-assump} on $\eps_N$, this is $O(N^{-1} \log^{-1}N)$.
Thus, for sufficiently large $N$, we obtain that
\[ \E [ S_{t+1} - S_t \,|\, \cF_t ]
  \le \frac{-S_t}{1+\lambda} \frac{N+1}{N} \frac{1}{N+3-X_t} + O(N^{-1})
  \le \frac{-S_t}{2(1+\lambda)N}, \]
completing the proof.
\end{proof}

The following estimate, that also follows easily from Lemma \ref{lem:drift-formula}, 
will be needed in Section \ref{sec:scaling-window}.

\begin{lemma}
\label{lem:drift-rough}
Assuming that $X_t = \frac{\lambda}{1+\lambda} N + b (1+o(1)) \sqrt{N \log N}$
with $b > 0$, we have 
\eqnspl{e:drift-rough}
{ \E [ \Delta S_t \,|\, \cF_t ]
&= -\frac{S_t}{N} \left( 1 + O \left( \sqrt{\frac{\log N}{N}} \right) \right) 
   + O(N^{-1}), }
as $N \to \infty$.
\end{lemma} 

\begin{proof}
From the assumption on $X_t$, we have 
$N+3-X_t = \frac{1}{1+\lambda} N \Big(1 - O \Big( \sqrt{\frac{\log N}{N}} \Big) \Big)$ (with $N+2-X_t$ having the same form). Substituting these into \eqref{e:drift-formula} yields the statement.
\end{proof}

We next compute a general formula for the conditional second moment of the increment of $S_t$.

\begin{lemma}
\label{lem:var-formula}
For all $t \ge 0$ we have
\eqnspl{e:var-formula}
{ &\E [ (S_{t+1} - S_t)^2 \,|\, \cF_t ] \\
  &\quad = \frac{1}{1+\lambda} \bigg[ \lambda + \frac{N+1}{N} \frac{\lambda^2}{N+2-X_t}
      + \frac{N+1}{N} \frac{2 \lambda}{N+2-X_t} \frac{-S_t + \lambda(N-X_t)}{N+3-X_t} \\
  &\qquad\qquad + 2 \frac{N+1}{N} 
      \frac{[-S_t + \lambda(N-X_t)] [-S_t + \lambda(N-X_t) + (N+3-X_t)]}{(N+3-X_t)(N+4-X_t)} \\
  &\qquad\qquad - \frac{N+1}{N} \frac{-S_t + \lambda(N-X_t)}{N+3-X_t} \bigg]. }
\end{lemma}

\begin{proof}
Similarly to the proof of \eqref{e:drift-formula}, we compute
\begin{align*}
 &\E [ (S_{t+1} - S_t)^2 \,|\, \cF_t ]
= \pi^{(0)}_{-1} + \sum_{k=0}^{X_t-Y_t} (\lambda+k)^2 \pi^{(-1)}_k 
+ \sum_{k=0}^{X_t-Y_t} k^2 \pi^{(0)}_k \\
&\quad = \frac{\lambda}{1+\lambda} + \lambda^2 \Pi^{(-1)}_0 + 2 \lambda \sum_{k=1}^{X_t-Y_t} k \pi^{(-1)}_k 
+ \sum_{k=1}^{X_t-Y_t} k^2 (\pi^{(-1)}_k + \pi^{(0)}_k) \\
&\quad = \frac{\lambda}{1+\lambda} + \lambda^2 \Pi^{(-1)}_0 
+ 2 \lambda \sum_{k=1}^{X_t-Y_t} k \pi^{(-1)}_k
+ \sum_{k=1}^{X_t-Y_t} k(k+1) (\pi^{(-1)}_k + \pi^{(0)}_k) \\
&\quad\qquad - \sum_{k=1}^{X_t-Y_t} k (\pi^{(-1)}_k + \pi^{(0)}_k).
\end{align*}
Rewriting the sums over $k$ in terms of the tails $\Pi^{(0)}_\ell$ and $\Pi^{(-1)}_\ell$, 
this equals
\begin{align}
\label{e:var-gen-form}
&\frac{\lambda}{1+\lambda} + \lambda^2 \Pi^{(-1)}_0 
+ 2 \lambda \sum_{\ell=1}^{X_t-Y_t} \Pi^{(-1)}_\ell 
+ 2 \sum_{\ell=1}^{X_t-Y_t} \ell (\Pi^{(-1)}_\ell + \Pi^{(0)}_\ell)
- \sum_{\ell=1}^{X_t-Y_t} (\Pi^{(-1)}_\ell + \Pi^{(0)}_\ell). 
\end{align}
For the sums over $\ell$ we use \eqref{e:sum-1st} and \eqref{e:sum-2nd} to get, in addition 
to the already obtained \eqref{e:ell-sum}, the relation
\eqnspl{e:ell-sum-ell}
{ \sum_{\ell=1}^{X_t-Y_t} \ell (\Pi^{(-1)}_\ell + \Pi^{(0)}_\ell)
&= \frac{1}{1+\lambda} \frac{N+1}{N} \sum_{\ell=1}^{X_t-Y_t} \ell \prod_{i=0}^{\ell-1}
\frac{X_t-Y_t-i}{N+2-Y_t-i} \\
&= \frac{1}{1+\lambda} \frac{N+1}{N} \frac{(X_t-Y_t) (N+3-Y_t)}{(N+3-X_t) (N+4-X_t)}. }
Substituting this into \eqref{e:var-gen-form} and then using \eqref{e:St-form}, we get the
statement. 
\end{proof}

\begin{proof}[Proof of Lemma \ref{lem:moments-new-gen}(iii).]
It is easy to check that each term in the right hand side of \eqref{e:var-formula} is bounded by
a constant dependent only on $\lambda$.
\end{proof}

The following estimate will be needed in Section \ref{sec:scaling-window}.

\begin{lemma}
\label{lem:var-rough}
Assuming that $X_t = \frac{\lambda}{1+\lambda} N + b (1+o(1)) \sqrt{N \log N}$ with $b > 0$, 
and that $|S_t| = O(\sqrt{N \log N})$, we have 
\eqn{e:var-rough}
{ \mathrm{Var} [ \Delta S_t \,|\, \cF_t ]
= 2 \lambda + O \left( \sqrt{\frac{\log N}{N}} \right), }
as $N \to \infty$.
\end{lemma}

\begin{proof}
From the assumption made on $X_t$, we have $N+4-X_t = \frac{1}{1+\lambda} N ( 1 - o(1) )$ (with 
$N+3-X_t$, etc.~having the same form). Hence the right hand side of \eqref{e:var-formula} equals 
\eqnsplst
{ &\frac{1}{1+\lambda} \left[ \lambda + O(N^{-1}) + O(N^{-1}) 
    + 2 \lambda (1+\lambda) + O ( (\log N)/N ) 
    - \lambda + O (\sqrt{(\log N)/N}) \right] \\
  &\quad = 2 \lambda + O \left( \sqrt{\frac{\log N}{N}} \right). }
Taking into account the already proven result Lemma \ref{lem:drift-rough}, the statement now follows.
\end{proof}

We next prove the last estimate from Lemma \ref{lem:moments-new-gen}.

\begin{proof}[Proof of Lemma \ref{lem:moments-new-gen}(iv).]
We write 
\eqnsplst{
 &E_N \Big[ (S_{t+1} - S_t)^2 e^{\eps_N (S_{t+1} - S_t)} \,\Big|\, \cF_t \Big] \\
 &\qquad = e^{-\eps_N} \pi^{(0)}_{-1} 
   + \sum_{k=0}^{X_t-Y_t} (\lambda + k)^2 e^{\eps_N (\lambda + k)} \pi^{(-1)}_k 
   + \sum_{k=0}^{X_t-Y_t} k^2 e^{\eps_N k} \pi^{(0)}_k \\
 &\qquad \le C(\lambda) 
   + C(\lambda) \sum_{k=1}^{X_t-Y_t} k^2 e^{\eps_N k} (\Pi^{(-1)}_k + \Pi^{(0)}_k), }
where in the last step we bounded $(\lambda + k)^2 \le C(\lambda) k^2$ and 
$\pi^{(i)}_k \le \Pi^{(i)}_k$, and we absorbed the $k=0$ terms into the additive constant 
$C(\lambda)$. Since $S_t \ge \eps_N N > 0$, in the formula of Lemma \ref{lemma:exactcomputation}
we have 
\[ \frac{X_t - Y_t - i}{N+2 - Y_t - i}
   \le \frac{X_t - Y_t}{N+2 - Y_t}
   = \frac{\frac{1}{1+\lambda} [ \lambda (N - Y_t) - S_t ]}{N + 2 - Y_t}
   \le \frac{\lambda}{1+\lambda} < 1. \]
Thus we have 
\[ \Pi^{(-1)}_k + \Pi^{(0)}_k
   \le \frac{1}{1+\lambda} \frac{N+1}{N} \left( \frac{\lambda}{1+\lambda} \right)^k, \]
and the claim follows.
\end{proof}
}

\section{Bounds on the support of the stationary measure}
\label{sec:scaling-window}

In this section we {establish asymptotic upper and lower bounds on the support of $\mu^{(N)}$ as $N\to\infty$ by proving Theorems~\ref{thm:scaling-lower} and~\ref{thm:scaling-upper} below}. Recall the constant {$a = \frac{ \sqrt{\lambda}}{1 + \lambda}$} from Section \ref{sec:constant} {and let $A$ denote the constant implicit in the deviation bound $\alpha_N=O(\sqrt{N\log N})$ of Theorem \ref{thm:main}.}

\begin{theorem}
\label{thm:scaling-lower}
For any $\eps > 0$ we have
\eqnst
{ \lim_{N \to \infty} \mu^{(N)} [0, \rho_c N + (a-\eps) \sqrt{N \log N} ]
= 0. }
\end{theorem}

\begin{theorem}
\label{thm:scaling-upper}
For any $\eps > 0$ we have
\eqnst
{ \lim_{N \to \infty} \mu^{(N)} [ \rho_c N + (a+\eps) \sqrt{N \log N}, \rho_c N + A \sqrt{N \log N} ]
= 0. }
\end{theorem}

{Together with Theorem~\ref{thm:main}, Theorems~\ref{thm:scaling-lower} and~\ref{thm:scaling-upper} imply our main result Theorem~\ref{thm:shift}.
To prove Theorem~\ref{thm:scaling-lower} and Theorem~\ref{thm:scaling-upper}, we need several auxiliary results, namely Lemmas \ref{lem:from-above}--\ref{lem:exit-big-2}
and Proposition \ref{prop:hit-0} below.}
As the proof is altogether quite lengthy, we first give an outline of the argument in the next section.

\subsection{Outline of the proof of Theorems~\ref{thm:scaling-lower} and \ref{thm:scaling-upper}.}
\label{ssec:outline}

The mean and variance of $\Delta S_t$ given by Lemmas \ref{lem:drift-rough} and \ref{lem:var-rough}
suggest that if space is rescaled by $1/\sqrt{N}$ and time is rescaled by $1/N$, then the process $S_t$ 
converges to an Ornstein-Uhlenbeck process over time-intervals of duration $O(N)$ (i.e.~rescaled 
time-intervals of duration $O(1)$). In particular, the "typical" deviation of $S_t$
from the "equilibrium" value $0$ should be $O(\sqrt{N})$, that is, the process $(X_t,Y_t)$
should stay, most of the time, within vertical distance $O(\sqrt{N})$ from the straight line 
of Figure \ref{fig:lx}. On the other hand, since it takes as much as $\Theta(N)$ time steps for $X_t$ 
to decrease by $1$, the process $S_t$ has "many chances" to make deviations of size $\gg \sqrt{N}$
from the "equilibrium" value $0$. The proof of our theorems amount to determining when exactly will
the probability of large downward deviations be significant enough to allow the process $Y_t$ to hit $0$.
The exact answer (i.e.~the value of the constant $a$ in the theorems) can be guessed by 
extrapolating the Ornstein-Uhlenbeck limit to deviations of $O(\sqrt{N \log N})$. 


Estimating the probability that $Y_t$ hits $0$ during a large downward deviation has the
following ingredients. First, assuming for a moment (for simplicity) that $X_t$ does not 
change at all, we have a hitting problem for the one-dimensional process $S_t/\sqrt{N}$.
We want to estimate the probability that starting with $S_{t_0} / \sqrt{N} \approx 0$, this
process exits an interval of the form $[-b \sqrt{\log N}, 1]$ at the left end-point. The same
exit probability for the limiting diffusion $\xi_s$ can be computed exactly, using that for a
suitable function $\varphi(x)$ the process $s \mapsto \varphi(\xi_s)$ is a martingale, giving 
\[ \mathbf{P}_0 [ \text{$\xi$ hits $-b \sqrt{\log N}$ before $1$} ]
   = \frac{\varphi(1) - \varphi(0)}{\varphi(1) - \varphi(-b \sqrt{\log N})}. \]
Asymptotics of the function $\varphi$ then yields that if $b < a(\lambda)$, then 
the exit probability is $N^{-1/2 + \eta}$ for some $\eta > 0$, where as if $b > a(\lambda)$
then it is $N^{-1/2-\bar{\eta}}$ for some $\bar{\eta} > 0$. As we will show, the number 
of chances for the discrete time process to make a large downward deviation is 
$N^{1/2} (\log N)^{\pm \Theta(1)}$ with high probability, so the value $a$ indeed marks 
the required transition point.

Suggested by the above heuristic, our strategy will be to show that 
$t \mapsto \varphi(S_t/\sqrt{N})$ is an \emph{approximate} martingale, allowing us to 
estimate the exit probability for the discrete time process. We will also need to take into 
account that $X_t$ (slowly) changes. We will need estimates on the time taken to exit the
interval $[-b \sqrt{\log N},1]$, and show that this is cannot be much shorter than $N$ or 
much longer than $N (\log N)^{O(1)}$. Finally, we also need to deal with estimates on how 
long it takes for $S_t$ to make upward excursion "away from equilibrium".

\subsection{Reducing the proof to some auxiliary statements}
\label{ssec:lemmas}

{\noindent\emph{Stopping time notation.} So far, we have worked exclusively with the canonical filtration
$\cF_t = \sigma(X_s, Y_s : 0 \le s \le t)$, $t \ge 0$. To lighten notation, we introduce the following conventions for the remainder of the paper:} if $t_1, \sigma_1$ are stopping times with 
respect to the canonical filtration (which will be the case unless stated otherwise), we abbreviate 
$\sigma_1(t_1) := \sigma_1 \circ \theta_{t_1}$, where $\theta_t$, $t \ge 0$ are the time-shift 
operators. In particular, $\sigma_1(t_1)$ is a stopping time with respect to the filtration 
$\cF_{t_1+t}$, $t \ge 0$.

Lemma \ref{lem:from-above} states that, no matter where the process is at a given time, the upward 
deviation will become smaller than a positive constant which does not depend on $N$, 
within $N \log^3 N$ steps, with very high (super-polynomial) probability.

\begin{lemma}
\label{lem:from-above} \ \\
Assume that $t_1$ is a stopping time at which $S_{t_1} > 4 \lambda + 2$. 
Let $\sigma_1 = \inf \{ t \geq 0 : S_{{t}} \le 4 \lambda + 2 \} $.
There exists $c > 0$ such that we have 
\[ \P [ \sigma_1(t_1) > N \log^3 N \,|\, \cF_{t_1} ] 
\le N^{-c \log N}. \]
\end{lemma}

Lemma \ref{lem:exit-big} below states that,
if the number of particles is $\rho_c N +  \Theta(\sqrt{N \log N}) $ and the upwards deviation is smaller than a certain constant, then within $N \log^3 N$ steps either the {$(X,Y)$-chain} will reach an absorbing state or the upward deviation will increase up to $O( \sqrt{N})$ with very high probability.

\begin{lemma}
\label{lem:exit-big} \ \\
(i) Assume that $t_2$ is a {stopping time} at which
\begin{equation}\label{eq:t2assumption}
\rho_c N + (a-2\eps) \sqrt{N \log N}
\le X_{t_2}
\le \rho_c N + (a-\eps) \sqrt{N \log N}. \end{equation}
and that we have
\[ S_{t_2} 
\le 4\lambda + 2. \]
{Let $\sigma_2 = \inf \{ t \ge 0 : \text{$S_{t} \ge N^{1/2}$} \}$ and 
{ recall that} $\tau = \inf \{ t \ge 0 : \text{$Y_{t} = 0$} \}$.
There exists $c > 0$ such that we have 
\[ \P [ (\sigma_2 \wedge \tau) (t_2) > N \log^3 N \,|\, \cF_{t_2} ] 
\label{e:exit-big}
\le N^{-c \log N}. \]}
(ii) The conclusion of part (i) remains valid if we assume that $t_2$ is a {stopping time} at which
\begin{equation}\label{eq:t2assumption-2}
\rho_c N + (a+\eps) \sqrt{N \log N}
\le X_{t_2}
\le \rho_c N + {(2 A+1)} \sqrt{N \log N}. \end{equation}
\end{lemma}

Lemma \ref{lem:exit-big-2} below complements Lemma \ref{lem:exit-big}(ii) by
providing an upper bound on the probability that the exit time is less than $b N$ 
for a suitable $b = b(\lambda) > 0$, assuming the starting position 
satisfies $S \approx 0$.

\begin{lemma}
\label{lem:exit-big-2} \ \\
Assume that $\bar{t}_2$ is a {stopping time} at which
\begin{equation}
\label{eq:bart2assumption-2}
\rho_c N + (a+\eps) \sqrt{N \log N}
\le X_{\bar{t}_2}
\le \rho_c N + {(2 A+1)} \sqrt{N \log N}, 
\end{equation}
and that for some $C = C(\lambda)$ we have
\eqn{e:assumptions-new}
{ |S_{\bar{t}_2}| 
  \le C. }
Let $\sigma_2 = \inf \{ t \ge 0 : \text{$S_{t} \ge N^{1/2}$} \}$ and 
let $\tau = \inf \{ t \ge 0 : \text{$Y_{t} = 0$} \}$.
There exists $b = b(\lambda) > 0$ such that we have 
\eqn{e:short-exit}
{ \P [ (\sigma_2 \wedge \tau) (\bar{t}_2) < b N \,|\, \cF_{\bar{t}_2} ] 
\le \frac{1}{2}. }

\end{lemma}


{Proposition \ref{prop:hit-0} below gives estimates for the probability, under the respective assumptions of Lemma \ref{lem:exit-big}(i) and Lemma \ref{lem:exit-big-2}, that the process will die out before `too many' particles are activated. It shows that there is a transition in the behaviour depending on whether $X_{t_2} < \rho_c N + (a-\eps) \sqrt{N \log N}$ or $X_{t_2} > \rho_c N + (a+\eps) \sqrt{N \log N}$.
}

\begin{proposition}[Conditional probability of stabilisation]
\label{prop:hit-0} \ \\
(i) Under the assumptions of Lemma \ref{lem:exit-big}(i) there exist $c > 0$, $\eta_1 = \eta_1(\eps,\lambda) > 0$ such that
\eqnst
{ \P [ \tau (t_2) < \sigma_2 (t_2) \,|\, \cF_{t_2} ] 
\ge c N^{-1/2+\eta_1}. } 
(ii) Under the assumption \eqref{eq:bart2assumption-2} on $\bar{t_2}$ and 
assuming that 
\eqn{e:assumptions}
{ S_{\bar{t}_2} 
  \le 4 \lambda + 2 \qquad \text{and} \qquad 
  Y_{\bar{t}_2} \ge (1+\lambda)(a+\eps/2)\sqrt{N \log N}, }

there exist $C > 0$, $\eta_2 = \eta_2(\eps,\lambda) > 0$ such that
\eqnst
{ \P [ \tau (\bar{t}_2) < \sigma_2 (\bar{t_2}) \,|\, \cF_{\bar{t}_2} ] 
\le C N^{-1/2-\eta_2}. } 
\end{proposition}

The rest of the paper is organised as follows. In Section \ref{ssec:iterate} we first prove Theorem \ref{thm:scaling-lower} assuming 
Lemmas \ref{lem:from-above} and \ref{lem:exit-big}(i) and
Proposition \ref{prop:hit-0}(i), and then prove Theorem 
\ref{thm:scaling-upper} assuming 
Lemmas \ref{lem:exit-big}(ii) and \ref{lem:exit-big-2} and Proposition \ref{prop:hit-0}(ii).
In Sections \ref{ssec:from-above} and \ref{ssec:hit-0} we prove the three lemmas and the proposition.

\subsection{Proof of Theorems~\ref{thm:scaling-lower} and~\ref{thm:scaling-upper}}
\label{ssec:iterate}

\begin{proof}[{Proof of Theorem \ref{thm:scaling-lower} assuming Lemmas
\ref{lem:from-above} and \ref{lem:exit-big} and Proposition \ref{prop:hit-0}(i).}]
{Let 
$$
s_0 := \inf \{ t \in \mathbb{N}_0 \, : \, X_t < \rho_c N + (a-\eps) \sqrt{N \log N}  \}.
$$
If $s_0 = \infty$, then   almost surely the process has
stabilized with at least $\rho_c N + (a-\eps) \sqrt{N \log N}$ particles, {an event that does not contribute to the probability mass to be bounded}.  
Assume now that $s_0 < \infty$.}
{Let us define the stopping time
\begin{align}
\sigma_1 & := \inf\{  t \geq 0 \, : \, S_{t} \leq 4 \lambda + 2  \}
\end{align}
and recall that
\eqnst
{ \sigma_2 = \inf\{  t \geq 0 \, : \, S_{t} \geq \sqrt{N}  \} \qquad \text{ and } \qquad
\tau = \inf\{  t \geq 0 \, : \, Y_{t} =0  \}. }}%
We first deduce from Lemma \ref{lem:from-above} that with probability at least 
$1 - N^{-c \log N}$ we have 
\[ s'_0 
:= s_0 + \sigma_1(s_0)
\le s_0 + N \log^3 N, \]
and {$S_{s'_0} \le 4\lambda+2$.} 

We now inductively define a sequence of stopping times 
$s'_0, \dots, s'_R$, with a random 
$0 \le R \le \lceil N^{1/2}/\log^4 N \rceil$, and a notion of `success'/`failure', as follows
(see Figure \ref{fig:s'_i} for an illustration of the scheme).
\begin{figure}
    \centering
    \setlength{\unitlength}{0.75cm}
    \begin{picture}(16.7,8)
        \put(0,2){\framebox(1.6,1.2){\parbox{1.2cm}{time $s'_i$ \\ $S \approx 0$}}}
        \put(0.8,3.4){\vector(1,1){0.75}}
        \put(2,0){\framebox(4,0.75){$\sigma_2 \wedge \tau > N \log^3 N$}}
        \put(6.2,0){$\longrightarrow$}
        \put(7.2,0){\parbox{3cm}{stop with $R=i$ \\ \textbf{failure}}}
        \put(0.8,1.8){\vector(1,-1){0.75}}
        \put(2,4){\framebox(4,0.75){$\sigma_2 \wedge \tau \le N \log^3 N$}}
        \put(4.8,4.95){\vector(1,1){1}}
        \put(6,6){\framebox(1.6,1.2){\parbox{1.2cm}{$\tau < \sigma_2$ \\ $Y = 0$}}}
        \put(7.8,6){$\longrightarrow$}
        \put(9,6){\parbox{3cm}{stop with $R = i$ \\ \textbf{success}}}
        \put(4.8,3.8){\vector(1,-1){1}}
        \put(6,2){\framebox(2,1.2){\parbox{1.5cm}{$\tau > \sigma_2$ \\ $S \approx \sqrt{N}$}}}
        \put(8.2,2.6){\vector(2,1){1.6}}
        \put(10,3){\framebox(3.2,0.75){$\sigma_1 > N \log^3 N$}}
        \put(13.3,3){$\longrightarrow$}
        \put(14.2,3){\parbox{3cm}{stop with $R=i$ \\ \textbf{failure}}}
        \put(8.2,2.4){\vector(2,-1){1.6}}
        \put(10,1){\framebox(3.2,0.75){$\sigma_1 \le N \log^3 N$}}
        \put(13.3,1){$\longrightarrow$}
        \put(14.2,1){\parbox{3cm}{reached time $s'_{i+1}$ \\ $S \approx 0$}}  
    \end{picture}
    \caption{The scheme for defining the stopping times $s'_i$ in the proof of Theorem \ref{thm:scaling-lower},
    with the cases leading to success/failure. Apart possibly from the initial index $i = 0$,
    we have $S \approx 0$ at time $s'_i$.}
    \label{fig:s'_i}
\end{figure}
The induction has been initiated with the definition of $s'_0$. Assuming $s'_i$ has been 
defined for some $0 \le i < N^{1/2}/\log^4 N$, we consider whether 
{$\sigma_2(s'_i) \wedge \tau(s'_i) > N \log^3 N$}. If yes, we set $R = i$, 
and declare that failure occurred. 
Otherwise, on the event {$\tau(s'_i) < \sigma_2(s'_i)$}
we set $R = i$ and say that success occurred. Further, on the event 
{$\tau(s'_i) > \sigma_2(s'_i)$},
we in addition consider $s''_i := s'_i + \sigma_2(s'_i)$, and whether 
{$\sigma_1(s''_i) > N \log^3 N$}.
If this happens, we set $R = i$ and declare that failure occurred. If not, we define 
{$s'_{i+1} := s''_i + \sigma_1(s''_i)$}. 
If $i+1$ reaches the value $\lceil N^{1/2}/\log^4 N \rceil$,
we stop the process with $R = \lceil N^{1/2}/\log^4 N  \rceil$ and declare that failure occurred.

Observe that by the weak law of large numbers, for any function $f(N) \to \infty$, we have {by Lemma \ref{lemma:exactcomputation} and by our assumptions}  
that during $N f(N)$ steps the $X$-coordinate moves to the left by $(1 \pm o(1)) f(N)$ with 
probability going to $1$. It follows that if for some $i < N^{1/2}/\log^4 N$ failure 
did not occur by step $i$, then we have, as $N$ becomes large, 
\[ {X_{ s'_i}} 
\ge \rho_c N + (a-\eps) \sqrt{N \log N} 
- \frac{1 + o(1)}{N} \frac{N^{1/2}}{\log^4 N} C N \log^3 N 
\ge \rho_c N + (a-2\eps) \sqrt{N \log N}. \]
Hence the bound of Lemma \ref{lem:exit-big} can be applied at the 
times $t_2 = s'_i$. Consequently, the probability that the process stops with a failure at some $R < N^{1/2}/\log^4 N$ is 
at most 
\[ \frac{N^{1/2}}{\log^4 N} C N^{-c \log N} =o(1). \]
On the other hand, by Proposition \ref{prop:hit-0}, the probability that the process continues 
for $N^{1/2}/\log^4 N$ steps without a success is at most 
\[ (1 - c N^{-1/2 + \eta_1})^{N^{1/2}/\log^4 N} =o(1). \]
On the event of a success, {we have that $(X,Y)$ hits $\{(x,0), x\in\mathbb{N}\}$ at a point $(x_0,0)$ satisfying
$x_0 \ge \rho_c N + (a-3\eps) \sqrt{N \log N}$,} so the theorem is proved.
\end{proof}


\begin{proof}[{Proof of Theorem \ref{thm:scaling-upper} assuming Lemmas \ref{lem:exit-big}(ii), \ref{lem:exit-big-2} and Proposition \ref{prop:hit-0}(ii).}]
{We may assume without loss of generality that $\eps < 1$ and that $a + 1 < A$.}
{Let 
$$
\bar{s}_0 := \inf \{ t \in \mathbb{N}_0 \, : \, X_t < \rho_c N + {(2A+1)} \sqrt{N \log N}  \}.
$$
If $\bar{s}_0 = \infty$, then almost surely the process has 
stabilized with more than $\rho_c N + {A} \sqrt{N \log N}$ particles
and the absorbing configuration does not contribute to the probability to be bounded.  
Assume now that $\bar{s}_0 < \infty$.}
It follows from Proposition \ref{prop:exit-} and the choice of $A$ in the proof of 
Theorem \ref{thm:main} that with probability tending to $1$, we have
\eqnst
{ S_{\bar{s}_0} 
\ge -(1+\lambda) A \sqrt{N \log N}. }
In the remainder of the proof we assume that this event occurs. Consequently,
\eqn{e:Y-true}
{ Y_{\bar{s}_0} 
= S_{\bar{s}_0} + (1+\lambda) \hat{X}_{\bar{s}_0} \sqrt{N \log N}
\ge (1+\lambda) A \sqrt{N \log N}
> (1+\lambda) (a+\eps) \sqrt{N \log N},}
where
\begin{equation}\label{eq:hatnotation}
\hat{X}_t = \frac{X_t - \rho_c N}{\sqrt{N \log N}}, t\in\mathbb{N}.
\end{equation}
{Let us define the stopping time
\begin{align*}
\bar{\sigma}_1 := \inf\{  t \geq 0 \, : \, S_{t} \leq 4 \lambda + 2  \}.
\end{align*}}%
We deduce from \eqref{e:Y-true} that taking
\eqnst{
\bar{t_2} 
= \bar{s}_0'
:= \bar{s_0} + \bar{\sigma}_1(\bar{s}_0) }
the conditions of Proposition \ref{prop:hit-0}(ii) are satisfied.

We now inductively define a sequence of stopping times 
$\bar{s}'_0, \dots, \bar{s}'_R$, with a random 
$0 \le \bar{R} \le N^{1/2} \log^3 N$, and a notion of `success'/`failure', as follows
(see Figure \ref{fig:bars'_i} for an illustration of the scheme).
\begin{figure}
    \centering
    \setlength{\unitlength}{0.75cm}
    \begin{picture}(18.5,8)
        \put(0,2){\framebox(1.6,1.2){\parbox{1.2cm}{time $\bar{s}'_i$ \\ $S \approx 0$}}}
        \put(0.8,3.4){\vector(1,1){0.75}}
        \put(2,0){\framebox(4.1,0.75){\parbox{3cm}{$\sigma_2 \wedge \tau > N \log^3 N$}}} 
        \put(6,0.25){$\longrightarrow$}
        \put(7,0.25){\parbox{3cm}{stop with $R=i$ \\ \textbf{failure}}}
        \put(0.8,1.8){\vector(1,-1){0.75}}
        \put(2,4){\framebox(4,0.75){\parbox{3cm}{$\sigma_2 \wedge \tau \le N \log^3 N$}}}
        \put(4.8,5.45){\vector(1,1){1}}
        \put(6,6.5){\framebox(1.6,1.2){\parbox{1.2cm}{$\tau < \sigma_2$ \\ $Y = 0$}}}
        \put(7.8,7){$\longrightarrow$}
        \put(9,7){\parbox{3cm}{stop with $R = i$ \\ \textbf{failure}}}
        \put(4.8,3.8){\vector(1,-1){1}}
        \put(6,2){\framebox(2,1.2){\parbox{1.5cm}{$\tau > \sigma_2$ \\ $S \approx \sqrt{N}$}}}
        \put(8.2,2.5){$\longrightarrow$}
        \put(9,2){\framebox(2.8,1.2){\parbox{2.1cm}{wait until $\bar{\sigma}_1$ \\ $S \approx 0$}}}
        \put(12,2.6){\vector(2,1){1}}
        \put(13.2,3){\framebox(2,0.75){$\hat{X} \le a+\eps$}}
        \put(15.2,3.2){$\longrightarrow$}
        \put(16,3.2){\parbox{3cm}{stop with \\ $R=i+1$ \\ \textbf{success}}}
        \put(12,2.4){\vector(2,-1){1}}
        \put(13.2,1){\framebox(2,0.75){$\hat{X} > a+\eps$}}
        \put(15.2,1){$\longrightarrow$}
        \put(16,1){\parbox{3cm}{reached $\bar{s}'_{i+1}$ \\ $S \approx 0$}}
    \end{picture}
    \caption{The scheme for defining the stopping times $\bar{s}'_i$ in the proof of 
    Theorem \ref{thm:scaling-upper}, with the cases leading to success/failure.
    Apart possibly from the initial index $i = 0$, we have $S \approx 0$ at time $\bar{s}'_i$.}
    \label{fig:bars'_i}
\end{figure}
The induction has been initiated with the definition of $\bar{s}'_0$. Assuming $\bar{s}'_i$ has been 
defined for some $i \ge 0$, we consider whether
$\sigma_2(\bar{s}'_i) \wedge \tau(\bar{s}'_i) > N \log^3 N$
occurs. If it does, we set $R = i$, and declare that failure occurred. 
Furthermore, on the event $\tau(\bar{s}'_i) < \sigma_2(\bar{s}'_i)$ we set $R = i$ and say that failure occurred. On the event $\tau(\bar{s}'_i) > \sigma_2(\bar{s}'_i)$,
we define $\bar{s}'_{i+1} := \bar{s}'_i  + \sigma_2(\bar{s}'_i) + \bar{\sigma}_1(\sigma_2(\bar{s}'_i))$. If $\hat{X}_{\bar{s}'_{i+1}}$ reaches a value smaller than $a+\eps$,
we stop the process with $R = i+1$ and declare that success occurred. Otherwise, the bound of Lemma \ref{lem:exit-big}(ii) can be applied at the 
time $\bar{t}_2 = \bar{s}'_{i+1}$. 

For each $1 \le i \le R$, let us call the index $i$ \textbf{bad}, if 
$\sigma_2(\bar{s}'_i) \wedge \tau(\bar{s}'_i) < b N$, where $b$ is the constant 
from Lemma \ref{lem:exit-big-2}, and let us call it \textbf{good} otherwise. 
By Lemma \ref{lem:exit-big-2} and the strong Markov property, each index is good
independently with probability $\ge 1/2$.

As in the proof of Theorem~\ref{thm:scaling-lower}, we use the weak law of large numbers to deduce that, for any function $f(N) \to \infty$, we have 
that during $N f(N)$ steps of the $(X,Y)$ Markov chain 
the $X$-coordinate moves to the left by $(1 \pm o(1)) f(N)$
with probability going to $1$. 
It follows from this that if by index $I = N^{1/2} \log^3 N$ of the
construction neither success 
nor failure occurred, and there have been at least $I/4$ good indices, then 
we would have 
\[ {X_{\bar{s}'_I}} 
\le \rho_c N + {(2 A+1)} \sqrt{N \log N} 
- \frac{1 + o(1)}{N} \frac{N^{1/2} \log^3 N}{4} b N
\le \rho_c N + (a+\eps) \sqrt{N \log N}, \]
a contradiction. By Lemma \ref{lem:exit-big-2}, the probability of there being
less than $I/4$ good indices goes to $0$.
Hence by Proposition \ref{prop:hit-0}(ii) and Lemma \ref{lem:exit-big},
the probability that the process stops with a failure at some 
$0 \le R < N^{1/2} \log^3 N$ is at most 
\[ N^{1/2} \log^3 N \big[ C N^{-1/2 - \eta_2} + C N^{-c \log N} \big] =o(1). \]

On the event of success, we have that {$(X,Y)$ did not hit $\{(x,0), x\in\mathbb{N}\}$ at a point $(x_0,0)$ with 
$x_0 \ge \rho_c N + (a+\eps) \sqrt{N \log N}$}, so the theorem is proved.
\end{proof}

\subsection{Hitting time and exit time estimates}
\label{ssec:from-above}

In Section \ref{sssec:hit-gen} we prove Lemma \ref{lem:from-above} 
via martingale arguments, and in Section \ref{sssec:exit-big-2} we prove Lemma \ref{lem:exit-big-2}. Sections \ref{sssec:smalljumps} 
and \ref{sssec:aux} provide auxiliary estimates. In Section   \ref{sssec:exit-big}
we prove Lemma \ref{lem:exit-big}.

\subsubsection{Martingale argument for positive deviations}
\label{sssec:hit-gen}

\begin{proof}[{Proof of Lemma \ref{lem:from-above}.}]
{Let $\sigma_1$ and $t_1$ be defined as in the statement of the lemma.
Due to the Strong Markov Property, we may replace $(X_0,Y_0)$ by the value of $(X_{t_1},Y_{t_1})$, 
and replace $\sigma_1(t_1)$ by $\sigma_1$. Note that then necessarily $S_0 > 4 \lambda + 2$.
Define 
\[ U_t 
= \log S_{\sigma_1 \wedge t}, \quad t \ge 0, \]
and note that $U_t$ is well-defined and positive for all $t \ge 0$, because on the event
$S_t > 4 \lambda + 2$ we have $S_{t+1} > 1$.}
We show that $U_t$ is a supermartingale, and give an upper bound on its drift.
On the event $\{ \sigma_1 > t \}$ we have
\eqnspl{e:Ut-drift}
{ \E [ U_{t+1} - U_t \,|\, \cF_t ]
&= \E \Big[ \log \Big( \frac{S_{t+1}}{S_{t}} \Big) \,\Big|\, \cF_{t} \Big] \\
&= \E \Big[ \log \Big( 1 + \frac{S_{t+1} - S_{t}}{S_{t}} \Big) \,\Big|\, \cF_{t} \Big] \\
&\le \E \Big[ \frac{S_{t+1} - S_{t}}{S_{t}} \,\Big|\, \cF_{t} \Big] = \frac{1}{S_{t}} \E [ S_{t+1} - S_{t} \,|\, \cF_{t} ]. }
Recalling the result of Lemma \ref{lem:drift-rough} we have 
\eqnsplst
{ \E [ S_{t+1} - S_{t} \,|\, \cF_{t} ] 
\le -\frac{S_{t}}{N} + o(N^{-1}). }
When $S_{t} > 4 \lambda + 2$, and $N$ is sufficiently large, the right hand side of the last expression gives
$
 \E [ S_{t+1} - S_{t} \,|\, \cF_{t} ] 
\le - \frac{1}{2} \frac{S_{t}}{N}. 
$
Substituting this back into \eqref{e:Ut-drift} we get
\eqnsplst
{ \E [ U_{t+1} - U_t \,|\, \cF_t ]
\le - \frac{1}{2N}. }
Define  {for each $t \geq 0$}
\[ V_t 
= U_t + {\frac{\sigma_1 \wedge t}{2N}}, \]
which is a non-negative supermartingale. The Optional Stopping Theorem gives that 
\[ \E [ V_{\sigma_1 \wedge {(4 N \log N)}} ]  
\le \E [ V_{0} ]
= \E [ U_{0} ]
= \E [ \log (S_{0}) ]
\le \log N. \]   
Since $U_t > 0$ for all $t \ge 0$,  the left hand side of the previous expression satisfies
\eqnsplst
{ \E [ V_{ {\sigma_1 \wedge (4 N \log N)}} ]
&= \E [ U_{\sigma_1 \wedge  (4 N \log N)} ] 
+ \frac{1}{2N} \E [ \sigma_1 \wedge (4 N \log N) ] \\
&\ge \frac{1}{2N} \E [ \sigma_1 \wedge {(4 N \log N)} ]. }
We conclude $\E [ \sigma_1 \wedge {(4 N \log N)} ] \le 2 N \log N$, 
and deduce using {Markov's inequality} that 
\[ \P [ \sigma_1 > 4 N \log N ] 
\le \frac{1}{2}. \]
To finish the proof, we iterate this bound $\frac{1}{4} \log^2 N$ times
{using the Strong Markov Property} to get 
\[ \P [ \sigma_1 > N \log^3 N ] 
\le \exp ( - \log 2 (1/4) \log^2 N ) 
= N^{-c \log N}. \]
\end{proof}

\subsubsection{Lower bound on the exit time}
\label{sssec:exit-big-2}

In this section we prove Lemma \ref{lem:exit-big-2}.

\begin{proof}[{Proof of Lemma \ref{lem:exit-big-2}.}]
We replace $(X_{\bar{t_2}},Y_{\bar{t_2}})$ by $(X_0,Y_0)$, 
$\sigma_2(\bar{t_2})$, $\tau(\bar{t_2})$ by $\sigma_2$, $\tau$.
We show that for sufficiently large $N$, with
\[ \tilde{\sigma}
   := \inf \{ t \ge 0 : |S_t| \ge \sqrt{N} \}, \]
the process 
$U_t := S_{t \wedge \tilde{\sigma}}^2 - 4 \lambda (t \wedge \tilde{\sigma})$
is a supermartingale. This indeed implies the claim, since using optional stopping, 
we have
\[ 
  C^2 
  \ge U_0
  \ge \E[ U_{bN} ]
  = \E [ S_{bN \wedge \tilde{\sigma}}^2 ] 
    - 4 \lambda \E [ b N \wedge \tilde{\sigma} ]
  \ge N \P [ \tilde{\sigma} < b N ] - 4 \lambda b N. \]
Rearranging yields
\[ \P [ \tilde{\sigma} < b N ]
   \le \frac{C^2 + 4 \lambda b N}{N}, \]
and choosing $b = b(\lambda) > 0$ sufficiently small gives the claim, since
$\{ \sigma_2 \wedge \tau < b N \} \subset \{ \tilde{\sigma} < b N \}$.

For the supermartingale property, using Lemmas \ref{lem:drift-rough} and 
\ref{lem:var-rough}, on the event $\{ |S_t| < \sqrt{N} \}$ 
we compute
\eqnsplst{
   &\E [ S_{t+1}^2 - S_t^2 \,|\, \cF_t ]
   = \E [ (S_{t+1} - S_t)^2 + 2 S_t (S_{t+1} - S_t) \,|\, \cF_t ] \\
   &\quad = \Var[S_{t+1} - S_t \,|\, \cF_t ] + \left( \E [ S_{t+1} - S_t \,|\, \cF_t ] \right)^2 + 2 S_t \E [ S_{t+1} - S_t \,|\, \cF_t ] \\
   &\quad = 2 \lambda + o(1) + \left( \frac{S_t}{N} (1+o(1)) + O(N^{-1}) \right)^2
      - 2 S_t \frac{S_t}{N} (1+o(1)) + O(|S_t|/N) \\
   &\quad \le 2 \lambda + o(1)
   \le 4 \lambda. }
This completes the proof.

\end{proof}

\subsubsection{No large jumps}
\label{sssec:smalljumps}
It will be useful to be able to ignore very large upwards jumps of $S_t$, and the 
next simple lemma states the required estimate for this.  We define  
$M_N := \log^2 N.$
\begin{lemma}
\label{lemma:smalljumps}
{Let $t_2$ be a stopping time satisfying the condition 
\[ \rho_c N + (a-2\eps) \sqrt{N \log N}
\le X_{t_2}
\le \rho_c N + {(2 A +1)}\sqrt{N \log N}. \]}%
\label{lem:large-up}
There exists some constant $c=c(\lambda)>0$ such that for any $t\geq t_2$
\[ \P [ \text{$\Delta S_t > M_N$ for some $t_2 \le t \le t_2 + \sigma_2(t_2) \wedge \tau(t_2) \wedge N \log^3 N$}  \,  \bigm |   \, {\mathcal{F}_{t_2}} ]
\le N^{-c \log N}. \]
\end{lemma}

\begin{proof}
We observe that
\[ S_{t+1} - S_t 
   = Y_{t+1} - Y_t - (1+\lambda) (X_{t+1} - X_t)
   \le Y_{t+1} - Y_t + (1+\lambda)
   \le Y_{t+1} - Y_t + \frac{1}{2} M_N, \]
for all large enough $N$. Thus $\Delta S_t > M_N$ implies $\Delta Y_t > \frac{1}{2} M_N$. 
From the formulas for the transition probabilities {in Lemma \ref{lemma:exactcomputation}} we obtain 
\begin{align*}
&\P [ Y_{t+1} - Y_t > \lfloor M_N/2 \rfloor \,|\, \cF_t ] 
= \Pi^{(0)}_{\lfloor M_N/2 \rfloor +1} + \Pi^{{(-1)}}_{\lfloor M_N/2 \rfloor +2} 
\le \Pi^{(0)}_{\lfloor M_N/2 \rfloor +1} + \Pi^{{(-1)}}_{\lfloor M_N/2 \rfloor +1} \\
&\quad = \frac{1}{1+\lambda} \frac{N+1}{N} \prod_{i=0}^{\lfloor M_N/2 \rfloor}
\frac{X_t - Y_t - i}{N+2 - Y_t - i} \Big[ \frac{1}{N+2 - X_t} 
+ \frac{N+1 - X_t}{N+2 - X_t} \Big] \\
&\quad \le \frac{1}{1+\lambda} \frac{N+1}{N} \Big( \frac{X_t}{N+2} \Big)^{M_N/2} 
\le \frac{1}{1+\lambda} \frac{N+1}{N} \Big( \frac{X_{t_2}}{N+2} \Big)^{M_N/2} \\
&\quad \le C(\lambda) e^{-c(\lambda) M_N} = C(\lambda) N^{-c(\lambda) \log N}. 
\end{align*}
The statement now follows {by a union bound.}
\end{proof}

\subsubsection{An auxiliary estimate}
\label{sssec:aux}

{ In analogy to \eqref{eq:hatnotation}, we set}
\eqn{e:Xhat-def}
{ \hat{x}=\hat{x}_N
= \frac{x - \rho_c N}{\sqrt{N \log N}},\quad x\in\mathbb{N},}
and further set $N' = N/\log^2 N$ and $x_* = x - \log^2 N$.

\begin{lemma}
\label{lem:X-diff}
The exists a finite constant $C$, such that for any $t,N\in\mathbb{N}$
\eqnst
{ \P [ X_{t+N'} < x_* \,|\, X_t = x,\, Y_t = y ] \le C N^{- \log N}. }
\end{lemma}

\begin{proof}
For $x - \log^2 N \le x' \le x$ we have
\eqnsplst
{ \Pi^{(-1)}_0(x',y)
&= \frac{1}{1+\lambda} \frac{N+1}{N} \frac{1}{N+2 - x'} 
= \frac{1}{N} (1+o(1)), }
where we used that 
\eqnst
{ N+2 - x'
= N - \frac{\lambda}{1+\lambda} N - \hat{x} \sqrt{N \log N} - O(\log^2 N) + 2
= \frac{1}{1+\lambda} N (1-o(1)). }
Thus the probability in the statement is bounded by the probability that a binomial random variable $B$ with $N'$ trials and probability of success $(1+o(1))/N$ exceeds $\log^2 N$.
The latter probability can be bounded
\eqnsplst
{ \P [ B > \log^2 N ]
&= \P [ e^{\theta B} > e^{\theta \log^2 N} ] \\
&\le N^{-\theta \log N} \Big( 1 + \frac{(e^\theta - 1)(1+o(1))}{N} \Big)^{N/\log^2 N} \\
&\le N^{-\theta \log N} \exp \Big( \frac{(e^\theta - 1)(1+o(1))}{\log^2 N} \Big) }
for any $\theta > 0$. Choosing $\theta = 1$ completes the proof.
\end{proof}

\subsubsection{Upper bound on the exit time}
\label{sssec:exit-big}

In this section we present the proof of  Lemma \ref{lem:exit-big}. 
As we outlined in Section \ref{ssec:outline},  we  will make use of some
standard functions associated with the diffusion limit of $S_t/\sqrt{N}$, and
we define these first, before starting the proof. Our intuition is that 
in the limit $N \to \infty$, the rescaled process $S_t/\sqrt{N}$ converges to 
the solution of the stochastic differential equation
\[ d \xi_s
   = - \xi_s \, ds + \sqrt{2 \lambda} \, dB_s, \]
where $B$ is a standard Brownian motion. (Background on SDEs can be 
found, for example, in \cite[Chapter 6]{Durrett}. Note, however, that neither
the discussion nor our proof will require familiarity with SDEs.) We define 
the function 
\begin{equation}\label{eq:varphifunction}
\varphi(x)
   = \int_0^x \exp \left( \frac{y^2}{2 \lambda} \right)\,\textup{d} y, \quad x\in \mathbb{R}, 
   \end{equation}
(called the natural scale of $\xi$). 
This function has the property that $\varphi(\xi_s)$ is a 
local martingale, and hence is a time-change of Brownian motion 
(see \cite[Section 6.1]{Durrett}). Therefore, our strategy will be to show 
that $\varphi(S_t/\sqrt{N})$ is an approximate martingale, that we can use, via
the optional stopping theorem, to prove Lemma \ref{lem:exit-big}
and  later 
Proposition \ref{prop:hit-0}.

We let 
\[ \alpha'_N := - (1+\lambda) b \sqrt{\log N}  \qquad\qquad \text{and} \qquad\qquad
   \beta'_N := 1, \]
and let 
\[ \alpha_N := \alpha'_N - 1 \qquad\qquad \text{and} \qquad\qquad
   \beta_N := \beta'_N+1. \]
(where the parameter $b$ will be chosen depending on the value of $X_{t_2}$).
There is a function $v_N$ continuous on $[\alpha_N,\beta_N]$ that is in
$C^2(\alpha_N,\beta_N)$, satisfying
\eqn{e:L-eqn}
{  L v_N(x)
   := - x v_N'(x) + \lambda v_N''(x)
   \equiv -1 \qquad\qquad \text{and} \qquad\qquad
   v_N(\alpha_N) = 0 = v_N(\beta_N). }
We extend $v_N$ to all of $\mathbb{R}$ by setting it equal to $0$ outside 
$[\alpha_N,\beta_N]$. (Here $L$ is the generator of $\xi$, and $v_N(x)$ gives the expected 
time $\xi$ takes to exit $(\alpha_N,\beta_N)$, starting at $x$.) We have 
the explicit formulas:
\eqnsplst{
 v_N(x)
 &= 2 \frac{\varphi(\beta_N) - \varphi(x)}{\varphi(\beta_N) - \varphi(\alpha_N)}
    \int_{\alpha_N}^x (\varphi(z) - \varphi(\alpha_N)) \frac{1}{2 \lambda} 
      e^{-\frac{z^2}{2 \lambda}} \, dz \\
 &\quad + 2 \frac{\varphi(x) - \varphi(\alpha_N)}{\varphi(\beta_N) - \varphi(\alpha_N)}
    \int_x^{\beta_N} (\varphi(\beta_N) - \varphi(z)) \frac{1}{2 \lambda} 
      e^{-\frac{z^2}{2 \lambda}} \, dz, }
and
\eqnsplst{ 
 v_N'(x)
 = \frac{\frac{1}{\lambda} e^{\frac{x^2}{2 \lambda}}}
   {\varphi(\beta_N) - \varphi(\alpha_N)} \int_{\alpha_N}^{\beta_N}
   e^{\frac{y^2}{2 \lambda}} \int_x^y e^{-\frac{z^2}{2 \lambda}} \, dz \, dy. }
and it is straightforward to verify that $v_N$ satisfies \eqref{e:L-eqn}.

We will need the following simple facts about $\varphi$ and $v_N$:
\begin{itemize}
    \item[(i)] $\varphi(-x) = - \varphi(x) 
      \sim - \frac{\lambda}{x} e^{\frac{x^2}{2 \lambda}}$, as $x \to \infty$.
    \item[(ii)] For all large enough $N$, the function $|v'_N(x)|$ is 
      maximized at $x = \alpha_N$, with 
      $v'_N(\alpha_N) \sim \frac{1}{\lambda} |\alpha_N| \log |\alpha_N|$. 
    \item[(iii)] We have 
      \eqnspl{e:v-deriv-bounds}{
      \sup_{x \in [\alpha_N,\beta_N]} |v_N'(x)| 
      &= O(\sqrt{\log N} \, \log \log N), \\
      \sup_{x \in [\alpha_N,\beta_N]} |v_N''(x)| 
      &= O(\log N \, \log \log N), \\
      \sup_{x \in [\alpha_N,\beta_N]} |v_N'''(x)| 
      &= O((\log N)^{3/2} \log \log N). }
\end{itemize}
The asymptotics in (i) can be verified using l'Hospital's rule. The asymptotics in 
(ii) can be derived from (i). The statement in (ii) about the maximum of 
$v_N'$ can be seen as follows. If $x_0 \in (\alpha_N,\beta_N)$ is such that 
$v_N''(x_0) = 0$, then \eqref{e:L-eqn} implies $x_0 v_N'(x_0) = 1$, and the 
formula for $v_N'$ implies that $x_0$ is bounded away from $0$. It is 
straightforward to check that $v_N'(\beta_N) = O(1)$, and this leaves 
$\alpha_N$ as the maximum point. Finally, (iii) follows from (ii) and 
\eqref{e:L-eqn}. 

A key role in the proof of Lemma \ref{lem:exit-big} will be played by the 
following computation. Let 
\[ V_t 
   = v_N(S_t/\sqrt{N}), \quad t \ge 0. \]

\begin{lemma}
\label{lem:v-super}
We have
\[ \E [ V_{t+1} - V_t \,|\, \cF_t ]
   = -\frac{1}{N} (1 + o(1)), \quad t \ge 0, \]
on the event when $S_t/\sqrt{N} \in [\alpha'_N,\beta'_N]$.
\end{lemma}

\begin{proof}
We estimate separately on the events $A := \{ S_{t+1} - S_t \le M_N \}$
and $B := \{ S_{t+1} - S_t > M_N \}$ (recall that $M_N = \log^2 N$). From 
Lemma \ref{lemma:smalljumps}(ii), we have that $\P [ B \,|\, \cF_t ] \le N^{-c \log N}$, 
and on this event we bound 
\[ |V_{t+1} - V_t| 
   \le 2 \sup_{x \in [\alpha_N,\beta_N]} v_N(x)
   \le C (\beta_N - \alpha_N) \sup_{x \in [\alpha_N,\beta_N]} |v'_N(x)|
   \le O( (\log N)^{3/2} ). \]
This yields
\[ \E [ (V_{t+1} - V_t) \mathbf{1}_{B} \,|\, \cF_t ] 
   = o(N^{-1}). \]

On the event $A$, we have $S_t/\sqrt{N},\, S_{t+1}/\sqrt{N} \in [\alpha_N,\beta_N]$.
Writing $s_t = S_t/\sqrt{N}$ for short, a Taylor expansion gives:
\eqnspl{e:Taylor}{ 
  V_{t+1} - V_t
  &= v_N'(s_t) \frac{S_{t+1} - S_t}{N^{1/2}}
     + \frac{1}{2} v_N''(s_t) \frac{(S_{t+1} - S_t)^2}{N}
     + \frac{1}{6} v_N'''(\zeta) \frac{(S_{t+1} - S_t)^3}{N^{3/2}}, }
where $\zeta$ is a number between $S_t/\sqrt{N}$ and $S_{t+1}/\sqrt{N}$.

We bound the absolute value of the third term in the 
right hand side of \eqref{e:Taylor} by $o(N^{-1})$. 
Taking conditional expectation, and using Lemmas \ref{lem:drift-rough}, we get
\eqnspl{e:V-Ito1}
{ \frac{v_N'(s_t)}{N^{1/2}} \E [ (S_{t+1} - S_t) \mathbf{1}_A \,|\, \cF_t ]
  &= \frac{v_N'(s_t)}{N^{1/2}} \Big[ \E [ S_{t+1} - S_t \,|\, \cF_t ] 
     - \E [ (S_{t+1} - S_t) \mathbf{1}_B \,|\, \cF_t ] \Big] \\
  &= \frac{v_N'(s_t)}{N^{1/2}} \left[ - \frac{S_t}{N} 
       + O \left( \frac{\sqrt{\log N}}{N^{3/2}} \right) 
     - \E [ (S_{t+1} - S_t) \mathbf{1}_B \,|\, \cF_t ] \right] \\
  &= - \frac{s_t v_N'(s_t)}{N} + o(N^{-1}). }
Similarly, using Lemma \ref{lem:var-rough}, for the second term in the right hand side of \eqref{e:Taylor} we have
\eqnspl{e:V-Ito2}
{ \frac{v_N''(s_t)}{2 N} \E [ (S_{t+1} - S_t)^2 \mathbf{1}_A \,|\, \cF_t ]
  &= \frac{v_N''(s_t)}{2 N} \Big[ \E [ (S_{t+1} - S_t)^2 \,|\, \cF_t ] 
     - \E [ (S_{t+1} - S_t)^2 \mathbf{1}_B \,|\, \cF_t ] \Big] \\
  &= \frac{v_N''(s_t)}{2 N} \left[ 2 \lambda + O \left( \frac{\sqrt{\log N}}{N^{3/2}} \right) 
     - \E [ (S_{t+1} - S_t)^2 \mathbf{1}_B \,|\, \cF_t ] \right] \\
  &= \frac{\lambda v_N''(s_t)}{N} + o(N^{-1}). }
Using \eqref{e:L-eqn}, the sum of the main terms can be written
\eqnsplst{
   \frac{1}{N} \left[ - s_t v_N'(s_t) + \lambda v_N''(s_t) \right] 
   = - \frac{1}{N}. }
This proves the claim.
\end{proof}

\begin{proof}[{Proof of Lemma \ref{lem:exit-big}.}]
(i) 
Recall the definitions of $t_2$, $\sigma_2$ and $\tau$ in the statement of the lemma. By the Strong Markov Property, we may replace $(X_0,Y_0)$ by the value of $(X_{t_2},Y_{t_2})$, and replace $\sigma_2(t_2)$ and $\tau(t_2)$ in the statement by $\sigma_2$ and $\tau$, respectively.

We first show that the event in the left hand side of \eqref{e:exit-big} implies 
the event
\[ \big\{ \sigma_2 \wedge \tau' > N \log^3 N \big\}, \]
where 
\[ \tau'
   := \inf \left\{ t \ge 0 : S_t \le - (1+\lambda)(a-\eps) \sqrt{N \log N} \right\}. \]
Indeed, if for all $0 \le t \le N \log^3 N$ we have $S_t < N^{1/2}$ and $Y_t > 0$, then 
for all such $t$ we also have
\eqnsplst{
   S_t 
   &= Y_t - (1+\lambda) X_t + \lambda N
   > -(1+\lambda) X_0 + \lambda N \\
   &> -(1+\lambda) \left[ \frac{\lambda}{1+\lambda} N 
     - (a-\eps) \sqrt{N \log N} \right] + \lambda N \\
   &= - (1+\lambda) (a-\eps) \sqrt{N \log N}. }
Take $b = a-\eps$ in the definition of $\alpha'_N$, so that
$\alpha'_N = -(1+\lambda) (a-\eps) \sqrt{\log N}$ in Lemma \ref{lem:v-super}. 
Hence we may replace $\tau$ by $\tau'$ in the inequality to be shown.

Observe that $S_0 / \sqrt{N} \in [\alpha'_N,\beta'_N]$. 
We define a process $U$ by setting, for each $t \geq 0$,  
\[ U_t
   := v_N(S_{t}/\sqrt{N}) + \frac{t}{2N}. \]
It follows from Lemma \ref{lem:v-super} that $U_{\sigma_2 \wedge \tau' \wedge t}$ 
is a supermartingale. Using optional stopping, we have
\eqnsplst{ 
   C_1 (\log N)^{3/2}
   &\ge U_0
   \ge \E [ U_{\sigma_2 \wedge \tau' \wedge t} ] \\
   &= \E [ V_{\sigma_2 \wedge \tau' \wedge t} ]
     + \frac{1}{2 N} \E [ \sigma_2 \wedge \tau' \wedge t ]
   \ge \frac{1}{2 N} \E [ \sigma_2 \wedge \tau' \wedge t ]. }
Letting $t \to \infty$, monotone convergence gives
\[ \E [ \sigma_2 \wedge \tau' ]
   \le 2 C_1 N \log^{3/2} N. \]
By Markov's inequality, we therefore have
\[ \P [ \sigma_2 \wedge \tau' > 4 C_1 N \log^{3/2} N ] 
   \le \frac{1}{2}. \]
Iterating this bound using the strong Markov property yields
\[ \P [ \sigma_2 \wedge \tau' > N \log^3 N ]
   \le (1/2)^{\lfloor \frac{1}{4 C_1} \log^{3/2} N \rfloor} 
   \le N^{-c \log^{1/2} N}. \] 

(ii) This can be proved in much the same way as part (i), by taking $b = (2A+1)$.
\end{proof}

\subsubsection{The key approximate martingale}

Recall the definition of $\varphi(x)$,
which is given in  (\ref{eq:varphifunction}).
It is straightforward to compute that 
\eqn{e:varphi-deriv}
{  \varphi'(x)
   = \exp \left( \frac{x^2}{2 \lambda} \right) \qquad \text{and} \qquad
   \varphi''(x) 
   = 2 x \varphi'(x) \frac{1}{2 \lambda}. }
This implies that $L \varphi \equiv 0$, where 
\[ L f (x) 
   := - x f'(x) + \lambda f''(x); \]
(that is, $L$ is the generator of $\xi$.) For technical reasons, it will be more
convenient to work with the following variant:
\[ \tilde{\varphi}(x)
   := \begin{cases}
      \varphi(2) - \varphi(x) & \text{$x \le 2$;} \\
      0                       & \text{$x > 2$.}
   \end{cases} \]
It is straightforward to check that $\tilde{\varphi}$ is non-negative, continuous, and
bounded away from $0$ on $(-\infty,1)$.

The martingale property we need is provided by the following lemma. Recall the assumptions 
of Lemma \ref{lem:exit-big}(i). 
Let us define the stopping time:
\[ \zeta
   := \inf \{ t \ge 0 : X_t - X_0 < - \log^3 N \}. \]

\begin{lemma}
\label{lem:varphi-mart}
Under the assumptions of Lemma \ref{lem:exit-big}(i), we have that there exists an
$(\cF_{t_2+t})_{t \ge 0}$-adapted process $e(t)$, such that 
\[ \E \Bigg[ \tilde{\varphi} \left( \frac{S_{t_2+t+1}}{\sqrt{N}} \right) 
      \,\Bigg|\, \cF_{t_2+t} \Bigg] 
      = \exp(e(t)) \tilde{\varphi} \, \left( \frac{S_{t_2+t}}{\sqrt{N}} \right). \]
Moreover, we have
\eqn{e:et-bound}
{  |e(t)|
   = O \left( \frac{\sqrt{\log N}}{N^{3/2}} \right) }
for all $0 \le t \le \sigma_2(t_2) \wedge \tau(t_2) \wedge \zeta(t_2)$.
Consequently, the process
\[ U_s
   := \tilde{\varphi}\left( \frac{S_{t_2+s}}{\sqrt{N}} \right)
      \exp \left( - \sum_{t=0}^{s-1} e(t) \right) \]
stopped at $\sigma_2(t_2) \wedge \tau(t_2) \wedge \zeta(t_2)$, is a martingale.
\end{lemma}

\begin{proof}
By the strong Markov property, we may replace $X_{t_2}$ by $X_0$, $S_{t_2+t}$ by 
$S_t$, etc. Put 
\[ V_t
   = \tilde{\varphi} \left( \frac{S_t}{\sqrt{N}} \right). \]    
We estimate $V_{t+1} - V_t$ separately on the events
\[ A 
   := \{ S_{t+1} - S_t > M_N \} \qquad \text{and} \qquad 
   B
   := \{ S_{t+1} - S_t \le M_N \}. \]
On the event $A$, we have $0 \le V_{t+1} < V_t$, and hence 
\[ \E \Bigg[ \left| \frac{V_{t+1}}{V_t} - 1 \right| \mathbf{1}_A \,\Bigg|\, \cF_t \Bigg] 
   \le \P [ A \,|\, \cF_t ]
   \le N^{-c \log N}. \]
On the event $B$, using the shorthand $x_0 = S_t/\sqrt{N}$, we have the Taylor expansion
\eqn{e:V_t-Taylor}
{  V_{t+1} - V_t
   = -\varphi' \left( x_0 \right) \frac{S_{t+1} - S_t}{\sqrt{N}}
     - \frac{1}{2} \varphi''(x_0) \frac{(S_{t+1} - S_t)^2}{N} 
     - \frac{1}{6} \varphi'''(z) \frac{(S_{t+1} - S_t)^3}{N^{3/2}}, }
where $z$ is value between $S_t/\sqrt{N}$ and $S_{t+1}/\sqrt{N}$.
We divide each term by $V_t = \tilde{\varphi}(x_0)$, to get
\eqn{e:V_t-Taylor-div}
{ \frac{V_{t+1}}{V_t} - 1
   = -\frac{\varphi' \left( x_0 \right)}{\tilde{\varphi}(x_0)} 
     \frac{S_{t+1} - S_t}{\sqrt{N}}
     - \frac{1}{2} \frac{\varphi''(x_0)}{\tilde{\varphi}(x_0)} 
     \frac{(S_{t+1} - S_t)^2}{N} 
     - \frac{1}{6} \frac{\varphi'''(z)}{\tilde{\varphi}(x_0)} 
     \frac{(S_{t+1} - S_t)^3}{N^{3/2}}. }
Now take conditional expectation with respect to $\cF_t$. 

The first term thus arising from \eqref{e:V_t-Taylor-div} we write as
\eqnspl{e:V_t-diff-1st}
{ & - \frac{\varphi'(x_0)}{\tilde{\varphi}(x_0)} \frac{1}{\sqrt{N}}
          \Big\{ \E [ S_{t+1} - S_t \,|\, \cF_t ] 
          - \E [ (S_{t+1} - S_t) \mathbf{1}_A \,|\, \cF_t ] \Big\}. }
For the second term in curly braces we use that 
\eqn{e:S_t-diff-triv}
{  |S_{t+1}-S_t| 
   = O(N) \qquad \text{ and } \qquad
   \P [ A \,|\, \cF_t ] \le N^{-c \log N}, }
to bound this term above by $N^{-1}$. Thus, using Lemma \ref{lem:drift-rough}, 
the expression in \eqref{e:V_t-diff-1st} equals
\eqnspl{e:V_t-diff-1st-2}
{  & - \frac{\varphi'(x_0)}{\tilde{\varphi}(x_0)} \frac{1}{\sqrt{N}}
          \left\{ - \frac{S_t}{N} \left( 1 + O \left( \frac{\sqrt{\log N}}{\sqrt{N}} \right) 
          \right) + O(N^{-1}) \right\} \\
   &\quad = \frac{x_0 \varphi'(x_0)}{N \tilde{\varphi}(x_0)}
      + \frac{\varphi'(x_0)}{\tilde{\varphi}(x_0)} O \left( \frac{\log N}{N^{3/2}} \right). }
In obtaining the second term, we used that $S_t/\sqrt{N} = O(\sqrt{\log N})$.

Similarly, the second term arising from \eqref{e:V_t-Taylor-div} we write as
\eqnspl{e:V_t-diff-2nd}
{ & - \frac{1}{2} \frac{\varphi''(x_0)}{\tilde{\varphi}(x_0)} \frac{1}{N}
        \Big\{ \E [ (S_{t+1} - S_t)^2 \,|\, \cF_t ] 
          - \E [ (S_{t+1} - S_t)^2 \mathbf{1}_A \,|\, \cF_t ] \Big\}. }
Again, for the second term in curly braces, we use \eqref{e:S_t-diff-triv} to bound 
that term above by $O(N^{-1/2})$. Thus, using Lemma \ref{lem:var-rough}, the 
expression in \eqref{e:V_t-diff-2nd} equals
\eqnspl{e:V_t-diff-2nd-2}
{ & - \frac{1}{2} \varphi''(x_0) \frac{1}{N \tilde{\varphi}(x_0)}
    \left[ 2 \lambda + O\left(\sqrt{\frac{\log N}{N}}\right) + O(N^{-1/2}) \right] \\
 &\quad = \frac{\varphi''(x_0) 2\lambda}{N \tilde{\varphi}(x_0)}
    + \frac{\varphi''(x_0)}{\tilde{\varphi}(x_0)} O \left( \frac{\log N}{N^{3/2}} \right). }

Finally, for the third term arising from \eqref{e:V_t-Taylor-div} we simply use that 
$|S_{t+1} - S_t| \mathbf{1}_B \le |S_{t+1} - S_t|$, to bound this term above by
\[ \sup_{z \in [-b(1+\lambda)\sqrt{\log N},2]} |\varphi'''(z)| \,
   O \left( N^{-3/2} \right). \]

The leading terms in \eqref{e:V_t-diff-1st-2} and \eqref{e:V_t-diff-2nd-2} cancel, 
due to \eqref{e:varphi-deriv}. Thus from \eqref{e:V_t-Taylor-div} we have 
\eqnspl{e:V_t-diff-red}
{ \Bigg| \E \bigg[ \frac{V_{t+1}}{V_t} - 1 \,\Bigg|\, \cF_t \bigg] \Bigg|
  &\le \frac{|\varphi'(x_0)|}{\tilde{\varphi}(x_0)} \, 
      O \left( \frac{\log N}{N^{3/2}} \right)
      + \frac{|\varphi''(x_0)|}{\tilde{\varphi}(x_0)} \, 
      O \left( \frac{\log N}{N^{3/2}} \right) \\
  &\quad\qquad + \frac{\sup_{z} |\varphi'''(z)|}{\tilde{\varphi}(x_0)} \,
      O \left( \frac{1}{N^{3/2}} \right). } 
Using the bounds \eqref{e:varphi-deriv} we get that the right hand side is 
$O ( N^{-3/2} \log^3 N )$, and this implies the claim about $e(t)$. The consequence
that $U_s$ is a martingale is immediate.
\end{proof}

\subsection{Estimates on the absorption time}
\label{ssec:hit-0}

In this section we prove Proposition \ref{prop:hit-0}(i)--(ii), using the martingale
defined in Lemma \ref{lem:varphi-mart}.

\begin{proof}[{Proof of Proposition \ref{prop:hit-0}(i).}]
By the strong Markov property, we may replace $t_2$ by time $0$, and assume that 
\[ \rho_c N + (a-2\eps) \sqrt{N \log N}
   \le X_0 \le \rho_c N + (a-\eps) \sqrt{N \log N}, \]
and $S_0 \le 4 \lambda + 2$. Let 
\[ \alpha_N 
   := -(1+\lambda) (a-\eps) \sqrt{\log N}, \] 
and let 
\[ \tau'
   := \inf \Big\{ t \ge 0 : S_t/\sqrt{N} \le \alpha_N \Big\}. \]
Here, in defining $\tau'$, we extend the dynamics of the Markov chain $(X_t,Y_t)$
beyond time $\tau$, by allowing the $Y$ coordinate to go negative, using the same transition 
probabilities as in Lemma \ref{lemma:exactcomputation}. 
We claim that $\{ \tau < \sigma_2 \} \supset \{ \tau' < \sigma_2 \}$. Indeed, if we have
$S_t/\sqrt{N} \le \alpha_N$, we also have
\eqnsplst{
   Y_t 
   &= S_t + (1+\lambda) X_t - \lambda N
   \le \alpha_N \sqrt{N} + (1+\lambda) X_0 - \lambda N \\
   &\le \alpha_N \sqrt{N} + (1+\lambda) 
       \left[ \rho_c N + (a-\eps) \sqrt{N \log N} \right] - \lambda N
   = 0. }
Therefore, it is sufficient to show a bound of the form 
$\P [ \tau' < \sigma_2 ] \ge c N^{-1/2+\eta_1}$.

By the optional stopping theorem, we have
\eqnspl{e:ost-U}{
  \varphi(2) - \varphi(S_0/\sqrt{N})
  &= U_0 
  = \E [ U_{\tau' \wedge \sigma_2 \wedge \zeta \wedge \lfloor N \log^3 N \rfloor} ] \\
  &= \E \left[ \tilde{\varphi} \left( \frac{S_{\tau' \wedge \sigma_2 \wedge \zeta 
    \wedge \lfloor N \log^3 N \rfloor}}{\sqrt{N}} \right) 
    \exp \left( - \sum_{t=0}^{\tau' \wedge \sigma_2 \wedge \zeta 
    \wedge \lfloor N \log^3 N \rfloor - 1} e(t) \right) \right]. }
Here the sum of the terms $e(t)$ is $O(N^{-1/2} \log^{7/2} N)$, due to \eqref{e:et-bound},
and hence the exponential is $1 + o(1)$. Note that the bound on the $o(1)$ term in this 
estimate is deterministic. We show that we can ignore
$\zeta$ inside the expectation. Due to Lemma \ref{lem:X-diff}, we have
$\P [ \zeta < \lfloor N \log^3 N \rfloor ] \le N^{-c \log N}$. On the subset 
$\{ \zeta < \tau' \wedge \sigma_2 \wedge \lfloor N \log^3 N \rfloor \}$ 
of this event, we have that $S_\zeta/\sqrt{N}$ falls in the interval 
$[\alpha_N,1]$, and on this interval, as $N \to \infty$, we have
\eqn{e:phi-sup-bnd}{
   \sup_{x \in [\alpha_N,1]} \tilde{\varphi}(x) 
   = \varphi(2) - \varphi(\alpha_N) 
   \sim \varphi(|\alpha_N|)
   \sim \frac{\lambda}{|\alpha_N|} \exp \left( \frac{\alpha_N^2}{2 \lambda} \right)
   = O (N^{1/2-\eta}), }
for some $\eta = \eta(\lambda,\eps) > 0$.
Thus we have that the right hand side of \eqref{e:ost-U} equals
\eqnspl{e:ost-U2}{
  &o(N^{1/2}) N^{-c \log N}
    + (1+o(1)) \E \left[ \tilde{\varphi} \left( \frac{S_{\tau' \wedge \sigma_2  
    \wedge \lfloor N \log^3 N \rfloor}}{\sqrt{N}} \right) \right] \\
  &\quad = o(1) + (1+o(1)) \Bigg\{ \E \left[ \tilde{\varphi}  
    \left( \frac{S_{\tau'}}{\sqrt{N}} \right)
    \mathbf{1}_{\tau' < \sigma_2,\, \tau' \le \lfloor N \log^3 N \rfloor} \right] \\
  &\quad\qquad + \E \left[ \tilde{\varphi}  \left( \frac{S_{\sigma_2}}{\sqrt{N}} \right)
    \mathbf{1}_{\sigma_2 < \tau',\, \sigma_2 \le \lfloor N \log^3 N \rfloor} \right] \\
  &\quad\qquad + \E \left[ \tilde{\varphi} 
    \left( \frac{S_{\lfloor N \log^3 N \rfloor}}{\sqrt{N}} \right)
    \mathbf{1}_{\lfloor N \log^3 N \rfloor < \tau' \wedge \sigma_2} \right] \Bigg\}. }
For the last term in the right hand side, we use that, due to Lemma \ref{lem:exit-big}(i), 
we have $\P [ \sigma_2 \wedge \tau' > N \log^3 N] \le N^{-c \log N}$, and using 
\eqref{e:phi-sup-bnd} this implies that this term is $o(1)$.
For the term involving $S_{\tau'}$, we use that the downward jumps of $S_t$ are of size
$1$, and therefore, $\alpha_N - N^{-1/2} \le S_{\tau'}/\sqrt{N} \le \alpha_N$. This implies
that this term is at most 
\[ (1+o(1)) \tilde{\varphi}(\alpha_N) \P [ \tau' < \sigma_2,\, \tau' \le N \log^3 N ]
   \le (1+o(1)) \tilde{\varphi}(\alpha_N) \P [ \tau' < \sigma_2 ]. \]
Finally, we need to estimate the term involving $S_{\sigma_2}$. For this, note that 
$S_{\sigma_2}/\sqrt{N} \ge 1$, and therefore 
$\tilde{\varphi}(S_{\sigma_2}) \le \varphi(2) - \varphi(1)$. Thus this term is 
at most 
\[ (\varphi(2) - \varphi(1)) (1 - \P [ \tau' < \sigma_2 ]). \] 
Altogether, from \eqref{e:ost-U}, \eqref{e:ost-U2} and the above estimates, we get that 
\eqnsplst{
   &\varphi(2) - \varphi(S_0/\sqrt{N}) \\
   &\quad \le (\varphi(2) - \varphi(1) ) ( 1 - \P [ \tau' < \sigma_2 ] ) 
       + o(1) + (1 + o(1)) (\varphi(2) - \varphi(\alpha_N)) 
       \P [ \tau' < \sigma_2 ]. }
Canceling the $\varphi(2)$ terms, and rearranging, yields:
\[ \varphi(-S_0/\sqrt{N}) - \varphi(-1) + o(1)
   \le \P [ \tau' < \sigma_2 ] \left( \varphi(|\alpha_N|)(1+o(1)) - \varphi(-1) \right). \]
As $N \to \infty$, the term in parentheses in the right hand side is positive, and hence,
using the bound $S_0 \le 2 \lambda + 4 = O(1)$, we get
\eqnsplst{
   \P [ \tau' < \sigma_2 ] 
   &\ge \frac{\varphi(-S_0/\sqrt{N}) - \varphi(-1) + o(1)}{\varphi(|\alpha_N|)(1+o(1)) - 
       \varphi(-1)}
   \ge \frac{\varphi(0) - \varphi(-1) + o(1)}{\varphi(|\alpha_N|)(1+o(1)) - \varphi(-1)} \\
   &\ge c \varphi(|\alpha_N|)^{-1}
   \ge c N^{-1/2+\eta}. }
\end{proof}

\begin{proof}[{Proof of Proposition \ref{prop:hit-0}(ii).}]
We follow a similar outline as in the proof of part (i). 
By the strong Markov property, we may replace $\bar{t}_2$ by time $0$, and assume that 
\[ \rho_c N + (a+\eps) \sqrt{N \log N}
   \le X_0 \le \rho_c N + (2A+1) \sqrt{N \log N}, \]
and that we have $S_0 \le 4 \lambda + 2$ and $Y_0 \ge (1+\lambda)(a+\eps/2)\sqrt{N \log N}$.
We will write 
\[ x_*
   = X_0 - \log^7 N, \]
so that by Lemma \ref{lem:X-diff}, we have
\[ \P \left[ X_{N \log^3 N} \le x_* \right]
   \le N^{- c \log N}. \]

We first find a suitable replacement for $\tau$. 
We will use the shorthand
\[ \tilde{X}_0 = \frac{X_0 - \rho_c N}{\sqrt{N}}. \]
The lower bound on $Y_0$ implies the following lower bound on $S_0$:
\eqnspl{e:S_0-lb}{
   S_0 
   &= Y_0 - (1+\lambda) X_0 + \lambda N \\
   &\ge (1+\lambda)(a+\eps/2)\sqrt{N \log N} 
       - (1+\lambda) \left[ \frac{\lambda}{1+\lambda} N 
       + \tilde{X}_0 \sqrt{N} \right]
       + \lambda N \\
   &= \sqrt{N} (1+\lambda) \left[ \left( a + \frac{\eps}{2} \right) \sqrt{\log N} 
      - \tilde{X}_0 \right] \\
   &\ge \sqrt{N \log N} (1+\lambda) \left[ a + \frac{\eps}{2} - 2A - 1 \right]. }
Put 
\[ \bar{\alpha}_N
   := - (1+\lambda) \left[ \tilde{X}_0 - \frac{\log^7 N}{\sqrt{N}} \right], \]
and let 
\[ \bar{\tau}'
   := \inf \Big\{ t \ge 0 : S_t/\sqrt{N} \le \bar{\alpha}_N \Big\}. \]
We claim that 
\[ \Big\{ \tau < \sigma_2,\, \tau \le \lfloor N \log^3 N \rfloor \Big\} 
   \subset \Big\{ \bar{\tau}' < \sigma_2,\, \bar{\tau}' \le \lfloor N \log^3 N \rfloor \Big\} 
           \cup \Big\{ X_{\lfloor N \log^3 N \rfloor} \le x_* \Big\}. \]
Indeed, if $Y_t = 0$ for some $0 \le t \le \lfloor N \log^3 N \rfloor$, but 
$X_{\lfloor N \log^3 N \rfloor} > x_*$, then we have
\eqnsplst{
   S_t
   &= Y_t - (1+\lambda) X_t + \lambda N
   \le - (1+\lambda) X_{\lfloor N \log^3 N \rfloor} + \lambda N \\
   &< - (1+\lambda) x_* + \lambda N 
   = - (1+\lambda) X_0 + \lambda N + (1 + \lambda) \log^7 N \\
   &= - (1+\lambda) \sqrt{N} \left[ \tilde{X}_0 - \frac{\log^7 N}{\sqrt{N}} \right]
   = \sqrt{N} \bar{\alpha}_N. }
Therefore, it is sufficient to show an upper bound of the form
\[ \P \Big[ \bar{\tau}' < \sigma_2,\, \bar{\tau}' \le \lfloor N \log^3 N \rfloor \Big]
   \le C N^{-1/2-\eta}. \]

By the optional stopping theorem, we have
\eqnspl{e:ost-U-bar}{
  \varphi(2) - \varphi(S_0/\sqrt{N})
  &= U_0 
  = \E [ U_{\bar{\tau}' \wedge \sigma_2 \wedge \zeta \wedge \lfloor N \log^3 N \rfloor} ] \\
  &= \E \left[ \tilde{\varphi} \left( \frac{S_{\bar{\tau}' \wedge \sigma_2 \wedge \zeta 
    \wedge \lfloor N \log^3 N \rfloor}}{\sqrt{N}} \right) 
    \exp \left( - \sum_{t=0}^{\bar{\tau}' \wedge \sigma_2 \wedge \zeta 
    \wedge \lfloor N \log^3 N \rfloor - 1} e(t) \right) \right]. }
Again, the sum of the terms $e(t)$ is $O(N^{-1/2} \log^{7/2} N)$, due to \eqref{e:et-bound},
and hence the exponential is $1 + o(1)$. We show that we can again ignore
$\zeta$ inside the expectation. Due to Lemma \ref{lem:X-diff}, we have
$\P [ \zeta < \lfloor N \log^3 N \rfloor ] \le N^{-c \log N}$. On the subset 
$\{ \zeta < \bar{\tau}' \wedge \sigma_2 \wedge \lfloor N \log^3 N \rfloor \}$ 
of this event, we have that $S_\zeta/\sqrt{N}$ falls in the interval 
$[\bar{\alpha}_N,1]$, and on this interval, as $N \to \infty$, we have
\eqn{e:phi-sup-bnd-bar}{
   \sup_{x \in [\bar{\alpha}_N,1]} \tilde{\varphi}(x) 
   \sim \varphi(|\bar{\alpha}_N|)
   = O (N^{O(1)}), }
since $\tilde{X}_0 = O (\sqrt{\log N})$.
Thus we have that the right hand side of \eqref{e:ost-U-bar} equals
\eqnspl{e:ost-U2-bar}{
  &O(N^{O(1)}) N^{-c \log N}
    + (1+o(1)) \E \left[ \tilde{\varphi} \left( \frac{S_{\bar{\tau}' \wedge \sigma_2  
    \wedge \lfloor N \log^3 N \rfloor}}{\sqrt{N}} \right) \right] \\
  &\quad = o(1) + (1+o(1)) \Bigg\{ \E \left[ \tilde{\varphi}  
    \left( \frac{S_{\bar{\tau}'}}{\sqrt{N}} \right)
    \mathbf{1}_{\bar{\tau}' < \sigma_2,\, \bar{\tau}' \le \lfloor N \log^3 N \rfloor} \right] \\
  &\quad\qquad + \E \left[ \tilde{\varphi}  \left( \frac{S_{\sigma_2}}{\sqrt{N}} \right)
    \mathbf{1}_{\sigma_2 < \bar{\tau}',\, \sigma_2 \le \lfloor N \log^3 N \rfloor} \right] \\
  &\quad\qquad + \E \left[ \tilde{\varphi} 
    \left( \frac{S_{\lfloor N \log^3 N \rfloor}}{\sqrt{N}} \right)
    \mathbf{1}_{\lfloor N \log^3 N \rfloor < \bar{\tau}' \wedge \sigma_2} \right] \Bigg\}. }
For the last term in the right hand side, we again use that, due to Lemma \ref{lem:exit-big}(i), 
we have $\P [ \sigma_2 \wedge \bar{\tau}' > N \log^3 N] \le N^{-c \log N}$, and using 
\eqref{e:phi-sup-bnd} this implies that this term is $o(1)$.
For the term involving $S_{\bar{\tau}'}$, we use that the downward jumps of $S_t$ are of size
$1$, and therefore, $\bar{\alpha}_N - N^{-1/2} \le S_{\bar{\tau}'}/\sqrt{N} \le \bar{\alpha}_N$. 
This implies that this term is at least 
\eqnsplst{ 
   &\tilde{\varphi}(\bar{\alpha}_N) \P [ \bar{\tau}' < \sigma_2,\, \bar{\tau}' \le N \log^3 N ] \\
   &\quad \ge \tilde{\varphi}(\bar{\alpha}_N) \left( \P [ \bar{\tau}' < \sigma_2 ] 
       - \P [ \sigma_2 \wedge \bar{\tau}' > N \log^3 N ] \right) \\
   &\quad \ge \tilde{\varphi}(\bar{\alpha}_N) \left( \P [ \bar{\tau}' < \sigma_2 ] 
       -  N^{- c \log N} \right) \\
   &\quad \ge (1+o(1)) \tilde{\varphi}(\bar{\alpha}_N) \P [ \bar{\tau}' < \sigma_2 ]. }
(Here we assumed that $\P [ \bar{\tau}' < \sigma_2 ] \ge N^{-(c/2) \log N}$, say, because
when this is not the case, the statement of the proposition holds automatically.)

Finally, we need to estimate from below the term involving $S_{\sigma_2}$. The main 
contribution will be that on the event $\sqrt{N} \le S_{\sigma_2} \le \sqrt{N} + M_N$ (recall that 
$M_N = \log^2 N$). On this event, we have the lower bound
\eqnsplst{
   &\tilde{\varphi} \left( 1 + \frac{M_N}{\sqrt{N}} \right) 
   \P \left[ \sigma_2 < \bar{\tau}',\, \sigma_2 \le \lfloor N \log^3 N \rfloor,\,
      1 \le \frac{S_{\sigma_2}}{\sqrt{N}} \le 1 + \frac{M_N}{\sqrt{N}} \right] \\
   &\quad = (1+o(1)) \tilde{\varphi}(1)
      \Big\{ \P \left[ \sigma_2 < \bar{\tau}' \right]
      - \P \left[ \sigma_2 \wedge \bar{\tau}' > \lfloor N \log^3 N \rfloor \right] \\
   &\quad\qquad - \P \left[ \text{$\exists$ $0 \le t \le \lfloor N \log^3 N \rfloor$: 
      $\Delta S_t > M_N$} \right] \Big\} \\
   &\quad = (1+o(1)) \tilde{\varphi}(1)
      \left\{ 1 - \P \left[ \bar{\tau}' < \sigma_2 \right]
      - N^{-c \log N} \right\} \\
   &\quad = (1+o(1)) \tilde{\varphi}(1)
      \left[ 1 + o(1) - \P \left[ \bar{\tau}' < \sigma_2 \right] \right]. }
On the event 
\[ \left\{ 1 + \frac{M_N}{\sqrt{N}} < \frac{S_{\sigma_2}}{\sqrt{N}} \right\} \]
we simply use that $\tilde{\varphi}(x) = \varphi(2) - \varphi(x) \ge 0$, as long as 
$x \le 2$, and that we defined $\tilde{\varphi}(x) = 0$ for $x > 2$, 
and lower bound the contribution on this event by $0$.

From \eqref{e:ost-U-bar}, \eqref{e:ost-U2-bar} and the above estimates, we get that 
\eqnsplst{
   &\varphi(2) - \varphi(S_0/\sqrt{N}) \\
   &\quad \ge (1+o(1)) (\varphi(2) - \varphi(1) ) ( 1 + o(1) - \P [ \bar{\tau}' < \sigma_2 ] ) 
       + o(1) \\
   &\quad\qquad + (1 + o(1)) (\varphi(2) - \varphi(\bar{\alpha}_N)) 
       \P [ \bar{\tau}' < \sigma_2 ]. }
Canceling the $\varphi(2)$ terms and rearranging yields:
\[ \varphi(-S_0/\sqrt{N}) - \varphi(-1) + o(1)
   \ge \P [ \bar{\tau}' < \sigma_2 ] \left( \varphi(|\bar{\alpha}_N|)(1+o(1)) - \varphi(-1) \right). \]
As $N \to \infty$, the term in parentheses in the right hand side is positive, and hence,
using the bound \eqref{e:S_0-lb} on $S_0$, we get
\eqnsplst{
   \P [ \bar{\tau}' < \sigma_2 ] 
   &\le \frac{\varphi(-S_0/\sqrt{N}) - \varphi(-1) + o(1)}{\varphi(|\bar{\alpha}_N|)(1+o(1)) - 
       \varphi(-1)} \\
   &\le \frac{\varphi \left( (1+\lambda) \left[ \tilde{X}_0 
       - \left( a + \frac{\eps}{2} \right) \sqrt{\log N} \right] \right) 
       - \varphi(-1) + o(1)}{\varphi(|\bar{\alpha}_N|)(1+o(1)) - \varphi(-1)}. }
We show that the right hand side is at most $c N^{-1/2-\eta}$,
which will complete the proof of the proposition. First, the terms $\varphi(-1)$ in the
numerator and denominator can be neglected, since both $|\bar{\alpha}_N| \to \infty$ and
\[ \tilde{X}_0 - \left( a + \frac{\eps}{2} \right) \sqrt{\log N} 
   \ge (a + \eps)\sqrt{\log N} - \left( a + \frac{\eps}{2} \right) \sqrt{\log N} 
   \to \infty. \]
On the one hand, in the denominator we have
\eqnsplst{
   \varphi(|\bar{\alpha}_N|) 
   \ge \varphi \left( (1+\lambda) \tilde{X}_0 \right)
   \sim \frac{\lambda}{(1+\lambda) \tilde{X}_0}
        \exp \left( \frac{(1+\lambda)^2}{2 \lambda} \tilde{X}_0^2 \right),
   \quad \text{as $N \to \infty$.} }
On the other hand, in the numerator we have
\eqnsplst{
  &\varphi \left( (1+\lambda) \left[ \tilde{X}_0 
       - \left( a + \frac{\eps}{2} \right) \sqrt{\log N} \right] \right) \\
  &\quad \sim \frac{\lambda}{(1+\lambda) \left[ \tilde{X}_0 
       - \left( a + \frac{\eps}{2} \right) \sqrt{\log N} \right]}
        \exp \left( \frac{(1+\lambda)^2}{2 \lambda} 
        \left( \tilde{X}_0 - \left( a + \frac{\eps}{2} \right) \sqrt{\log N} \right)^2 \right). }
The ratio of the factors in front of the exponentials is bounded by $C(\eps)$.
Using that $(1+\lambda)^2 / 2 \lambda = 1/2a^2$, the ratio of the exponentials equals
\[ \exp \left( \frac{1}{2 a^2} \left[ - 2 \tilde{X}_0 
     \left( a + \frac{\eps}{2} \right) \sqrt{\log N}
     + \left( a + \frac{\eps}{2} \right)^2 \log N \right] \right). \]
Using that $\tilde{X}_0 \ge (a+\eps) \sqrt{\log N}$, the right hand side is at most
\[ \exp \left( - \frac{\log N}{2} \frac{(a+\eps/2)^2}{a^2} \right) 
   \le N^{-1/2-\eta(\eps,\lambda)}. \]
\end{proof}

\section{Appendix}

\subsection{Proof of Lemma \ref{lem:sums}}
\label{ssec:proof-sums}

\begin{proof}[Proof of Lemma \ref{lem:sums}]
(i) We begin by showing \eqref{e:sum-1st}. Adding the $\ell = 0$ term (equal to $1$) to the summation, we have the sum
\begin{align*}
 \sum_{\ell=0}^n \frac{n!}{(n-\ell)!} \frac{(m-\ell)!}{m!}
&= \frac{1}{\binom{m}{n}} \sum_{\ell=0}^n \binom{m-\ell}{n-\ell} \\
&= \frac{1}{\binom{m}{n}} \sum_{\ell=0}^n \binom{m-n+\ell}{\ell} \\
&= \frac{\binom{m+1}{n}}{\binom{m}{n}} = \frac{m+1}{m-n+1}. 
\end{align*}
Here the penultimate step uses a well-known identity for binomial coefficients that is easy to verify. Moving to the proof of \eqref{e:sum-2nd}, we obtain by similar manipulations
\begin{align*}
\sum_{\ell=0}^n \ell \frac{n!}{(n-\ell)!} \frac{(m-\ell)!}{m!}
&= \frac{1}{\binom{m}{n}} \sum_{\ell=0}^n \ell \binom{m-\ell}{n-\ell}  \\ & = \frac{1}{\binom{m}{n}} \sum_{\ell=0}^n (n-\ell) \binom{m-n+\ell}{\ell} \\
&= \frac{n \binom{m+1}{n}}{\binom{m}{n}} 
- \frac{1}{\binom{m}{n}} \sum_{\ell=1}^n \ell \binom{m-n+\ell}{\ell} \\
&= \frac{n (m+1)}{m-n+1} 
- \frac{1}{\binom{m}{n}} \sum_{\ell=1}^n (m-n+1) \binom{m-n+\ell}{\ell-1} \\
&= \frac{n (m+1)}{m-n+1} 
- \frac{m-n+1}{\binom{m}{n}} \sum_{\ell=0}^{n-1} \binom{m-n+1+\ell}{\ell} \\
&= \frac{n (m+1)}{m-n+1} 
- \frac{(m-n+1) \binom{m+1}{n-1}}{\binom{m}{n}} \\
&= \frac{n (m+1)}{(m-n+1) (m-n+2)}. 
\end{align*}
\end{proof}



\printbibliography









\paragraph*{Acknowledgments.}
We thank Matthew Junge for pointing out the reference \cite{kaufman2025asymptoticbehaviorcriticaldensity}.


\end{document}